\documentclass[12pt]{amsart}
\usepackage{amssymb,epsfig,color}
\usepackage{amsmath}
\usepackage{amscd}
\usepackage{amsthm, color}

\usepackage{enumerate}
\newtheorem{thm}{Theorem}[section]
\newtheorem{lem}[thm]{Lemma}
\newtheorem{prop}[thm]{Proposition}
\newtheorem{cor}[thm]{Corollary}
\theoremstyle{definition}
\newtheorem{defn}[thm]{Definition}
\newtheorem{rmk}[thm]{Remark}

\newcommand{\ch}{\operatorname{ch}}
\newcommand{\ind}{\operatorname{ind}}

\newcommand{\Endo}{\operatorname{End}}
\newcommand{\End}{\operatorname{End}}

\newcommand{\supp}{\operatorname{supp}}

\newcommand{\Diff}{\operatorname{Diff}}
\newcommand{\psos}{pseudodifferential operators}

\newcommand{\psimult}{\Psi^0_{mult}}

\newcommand{\pa}{\partial}
\newcommand{\psia}{\Psi^0_{a}}

\newcommand{\into}{\longrightarrow}
\renewcommand{\hat}{\widehat}
\renewcommand{\tilde}{\widetilde}

\newcommand{\NN}{\mathbb{N}}
\newcommand{\ZZ}{\mathbb{Z}}

\newcommand{\RR}{\mathbb{R}}
\newcommand{\CC}{\mathbb{C}}

\def\maB{\mathcal{B}} 
\def\maC{\mathcal{C}} 
\def\maD{\mathcal{D}}            
            
\def\maF{\mathcal{F}} 
\def\maG{\mathcal{G}}

\def\maK{\mathcal{K}}

\def\maR{\mathcal{R}}

\def\maV{\mathcal{V}} 
\def\maW{\mathcal{W}}

\newcommand\Dira{\hspace*{.3mm} / \hspace*{-2.5mm}D}
\newcommand\Dir{/\hspace*{-2.8mm}D}

\newcommand\vol{\operatorname{vol}}

\newcommand\CIc{\mathcal{C}^\infty_{c}}
\newcommand\CI{\mathcal{C}^\infty}
\newcommand\ie{{\em i.e., }}

\title{An index formula for perturbed Dirac operators on Lie manifolds}

\begin{document}

\author[C. Carvalho]{Catarina Carvalho}
       \address{Departamento de Matem\'atica, Instituto Superior
         T\'ecnico, UTL, Lisbon, Portugal}
       \email{ccarv@math.ist.utl.pt}

\author[V. Nistor]{Victor Nistor} \address{Pennsylvania State
       University, Math. Dept., University Park, PA 16802, USA, and
        Inst. Math. Romanian Acad.  PO BOX 1-764, 014700 Bucharest
        Romania} \email{nistor@math.psu.edu}

\date\today

\thanks{Carvalho partialy supported by Funda\c{c}\~{a}o para a Ci\^{e}ncia e 
Tecnologia through the grant FCT POCI/MAT/55958/2004.  Carvalho's
manuscripts are available from {\bf
http:{\scriptsize//}www.math.ist.utl.pt/~ccarv}. Nistor was partially
supported by the NSF Grants DMS-0713743, OCI-0749202, and
DMS-1016556. Manuscripts available from {\bf
http:{\scriptsize//}www.math.psu.edu{\scriptsize/}nistor{\scriptsize/}}.}





\begin{abstract} 
We give an index formula for a class of Dirac operators coupled with
unbounded potentials.  More precisely, we study operators of the form
$P := \Dira + V$, where $\Dira$ is a Dirac operators and $V$ is an
unbounded potential at infinity on a possibly non-compact manifold
$M_0$.  We assume that $M_0$ is a Lie manifold with compactification
denoted $M$. Examples of Lie manifolds are provided by asymptotically
Euclidean or asymptotically hyperbolic spaces. The potential $V$ is
required to be such that $V$ is invertible outside a compact set $K$
and $V^{-1}$ extends to a smooth function on $M \smallsetminus K$ that
vanishes on all faces of $M$ in a controlled way. Using tools from
analysis on non-compact Riemannian manifolds, we show that the
computation of the index of $P$ reduces to the computation of the
index of an elliptic pseudodifferential operator of order zero on
$M_0$ that is a multiplication operator at infinity. The index formula
for $P$ can then be obtained from the results of \cite{carvalho}. The
proof also yields similar index formulas for Dirac operators coupled
with bounded potentials that are invertible at infinity on
asymptotically commutative Lie manifolds, a class of manifolds that
includes the scattering and double-edge calculi.
\end{abstract}

\maketitle

\tableofcontents

\section*{Introduction}\label{s int}

Perturbed Dirac operators $\Dir + V$ and operators $\Delta + V$ of
Schr\"{o}dinger type on non-compact manifolds play an important role
in Quantum Mechanics, Conformal Field Theory, and in other
areas. Partly for this reason, the index theory for this kind of
operators has been the subject of extensive research \cite{Anghel1,
Anghel2, BottSee, BM, bunke, Callias, FoxHaskell1, FH2, FH3, Kottke, rade}.

The purpose of this paper is to give an index formula for Dirac
operators operators coupled with unbounded potentials on
even-dimensional Lie manifolds, a class of non-compact manifolds $M_0$
whose structure at infinity is controlled by a Lie algebra of vector
fields tangent to the boundary of a suitable (given) compactification
$M$. We also find an index formula for operators coupled with bounded
potentials on a subclass of Lie manifolds that are commutative at
infinity (see Definition \ref{def.com.inf}).

Lie manifolds, or manifolds with a Lie structure at infinity, were
introduced and studied in \cite{ALNgeo}. There is a natural algebra of
differential operators associated to any such manifold that contains
all the classical geometric operators, such as the the Dirac
operator \cite{aln2}. One also defines a suitable algebra of \psos\ on
any Lie manifold \cite{ALNpdo}, which happens to be related to an
algebra of \psos\ on a differentiable groupoid. For many of these
algebras $\Psi^*$ of \psos\ on manifolds with corners, the
Fredholmness of $P \in M_n(\Psi^*)$ can be characterised by the
invertibility of a symbol class that consists of two components: the
principal symbol $\sigma_{0}(P)$ and a symbol at the boundary
$\sigma_{\pa}(P)$, also called the {\em indicial operator} associated
to $P$. Thus a pseudodifferential operator compatible with the Lie
manifold structure is Fredholm if, and only if, the following two
conditions are satisfied: the usual ellipticity and the invertibility
in the so-called indicial algebra at the boundary. The Fredholm
conditions relevant for our case are discussed in
Propositions \ref{prop.fredholm.lie} and
\ref{prop.Fred.comm}.  

Let $M_0$ be an even-dimensional Riemannian Lie manifold, with
compactification to a manifold with corners $M$ and $\maV$ be the
Lie algebra of vector fields tangent to the faces of $M$ and defining
the structure at infinity of $M_0$ (for precise definitions
see \S.\ref{s lie mfld}). Let $W$ be a {Clifford module} over $M$
endowed with an admissible connection and let
$\Dir: \maC^\infty(M;W) \to \maC^\infty(M;W)$ be the associated
generalized Dirac operator.  Let us denote by $\{x_k\}$ the boundary
defining functions of the hyperfaces of $M$. We shall consider
operators of the form
\begin{equation}
\label{eq.Dirac}
    {T'} = \Dir + V := {\Dir} \hat{\otimes} 1 + 1 \hat{\otimes}
    V: \maC^\infty_c(M_0;W\otimes E) \to \maC_c^\infty(M_0;W\otimes
    E),
\end{equation}
where the potential $V\in \End(E)$ is of the form $V=f^{-1}V_0$ with 
\begin{equation}\label{eq.def.f}
        f := \Pi x_k^{a_k}, \quad  a_k \in \ZZ,\ a_k > 0,
\end{equation}
and $V_0$ smooth on $M$ and invertible at infinity (that is, on $\pa
M$). We prove that $T'$ is essentially self-adjoint acting on
$L^2(M_0; W \otimes E)$. We shall denote by $T$ the closure of $T'$, which is hence a self-adjoint operator
(odd with respect to the natural spinor grading). Let $\maD(T)$ denote
the domain of $T$ and $\maD(T) = \maD(T)_+ \oplus \maD(T)_-$ be its
grading. We shall still write $T = \Dir + V$, for simplicity. Let $ W \hat \otimes E$ be the tensor product 
$W \otimes E$ endowed with the usual grading.

Our main result, Theorem \ref{thm.u.potential}, is an index formula
for the chiral operator
\begin{equation*}
        T_+ : \maD(T)_+ \to L^2(M_0; W \hat \otimes E)_-
\end{equation*}
similar to the usual Atiyah-Singer index formula. The proof of this
theorem is obtained from a sequence of reductions, ultimately reducing
our main result to the Atiyah-Singer type theorem for operators that
are asymptotically multiplication at infinity \cite{carvalho}.  Let us
mention that our Theorem \ref{thm.u.potential} is about as general as
one can hope for such that a classical index formula would still
apply. For instance, if one replaces $V$ with a {\em bounded}
potential $V_0$, then one expects an index formula for $\Dir + V_0$ to
involve non-local invariants similar to the eta invariant \cite{APS1}.
These non-local invariants would be associated to the faces at
infinity. Thus, if one wants to avoid non-local invariants and have an
index formula on an arbitrary Lie manifold just in terms of classical
Chern characters, then one needs to require $V$ to be unbounded at
infinity. (Note, however, that on asymptotically commutative Lie
manifolds, Definition \ref{def.com.inf}, we do allow bounded
potentials, and the calculation in this case is an important
ingredient in the proof; see below for more details.) Moreover,
imposing some structure at infinity also seems to be necessary and is
usually done in practice. This justifies why we consider Lie manifolds
and not more general non-compact manifolds. See \cite{bunke} and
\cite{Kottke} for some related approaches.

Most of the known results on the index of perturbed Dirac operators on
non-compact manifolds cited above make use of crucial properties of
Dirac operators, namely relative index theorems, trace formulas, or
boundary conditions. In this paper, our index formula for $\Dir + V$,
with $V$ bounded, is obtained from a general index theorem for a
suitable class of pseudodifferential operators and in fact most of our
results hold in the setting of pseudodifferential operators. For
bounded potentials, however, we need to assume that our Lie manifold
is asymptotically commutative (or commutative at infinity),
Definition \ref{def.com.inf}. A similar approach, in the bounded
potentials case, is contained in \cite{Kottke}, for odd-dimensional
manifolds, where Melrose's index formula for (families of) scattering
operators is used to derive an index theorem for perturbed
pseudodifferential operators, so-called Callias-type operators, with
{bounded} potentials. It is shown there that the index can be computed
from invariants at the corner $S_{\pa M}^*M$.  Note that the
scattering structure is just a particular case of the asymptotically
commutative Lie structures we consider here. (It is easily seen that
Theorem \ref{ind.thm.psia} extends to the case of families, so our
results can also be formulated in this setting.)  To get the result
for unbounded potentials, we need harder results from analysis, so we
stick to differential operators, but this results holds for arbitrary
Lie manifolds.

Let us now review the sequence of reductions that lead to Theorem
\ref{thm.u.potential}. At the same time, we will review the contents of the
paper (but in the inverse order of the sections). The first step is to
write
\begin{equation*}
  T = \Dir + V = f^{-1/2} Q f^{-1/2}, \quad \mbox { with } Q :=
  f^{1/2} \Dir f^{1/2} + V_0,
\end{equation*}
which we show in Section \ref{sec3} to have the same index as 
$T$. We then consider a new Lie manifold structure
$(M, \maW)$ on $M_0$ using
\begin{equation}\label{eq.def.W}
        \maW = f\maV := (\Pi x_k^{a_k})\maV, 
\end{equation}
It turns out that $Q \in \Diff_{\maW}(M)$, which justifies the
introduction of the new Lie manifold structure $(M, \maW)$. Moreover,
$Q$ itself is a Dirac operator coupled with the {\em bounded}
potential $V_0$. What makes the index of such an operator computable
is the fact that the structural Lie algebra of vector fields $\maW$
defining $Q$ is commutative at infinity, or, to put this in another
way, the indicial algebra of $\maW$ is commutative.  A Lie manifold
with this property will be called {\em asymptotically
commutative}. The analysis on the new Lie $(M,\maW)$ manifold turns
out to be much easier. The index of the operator $Q = \Dir + V_0$
associated to (general) asymptotically commutative Lie manifolds with
$V_0$ bounded, but invertible at infinity, is obtained in
Theorem \ref{thm.b.potential} using results of Section \ref{sec2}.

The analytical properties of general, not necessarily
even-dimensional, asymptotically commutative Lie manifolds $(M, \maW)$
and the index of operators on these spaces are studied in
Section \ref{sec2}. We show that fully elliptic operators in
$\Psi^*(M; \maW)$ can be deformed continuously to operators in
$\Psi^*(M; \maW)$ that are {\em asymptotically multiplication} on
$M_0$. We thus obtain an index theorem for fully elliptic
pseudodifferential operators on general asymptotically commutative Lie
manifolds (that is, of the form $(M, \maW)$ with $\maW$ commutative at
infinity), Theorem \ref{ind thm lie mflds}, generalizing known results
for the scattering and double edge operators \cite{LMor2,
MelroseScattering}. Let us also mention that the case of
asymptotically commutative Lie manifolds includes the important case
of asymptotically Euclidean manifolds. The index formula for fully
elliptic operators on asymptotically commutative Lie manifolds follows
then from the results of \cite{carvalho}. This reduction is achieved
in Section \ref{sec1}. We remark that all the results in
Sections \ref{sec1} and \ref{sec2} do not assume that our manifolds
are even-dimensional.

It is a classical result that on a compact manifold $M_1$, a pseudodifferential operator $P$ of
order $m$ defines a Fredholm operator $H^s(M_1) \to H^{s-m}(M_1)$ if,
and only if, it is elliptic.  In other words, on a compact manifold,
ellipticity is equivalent to Fredolmness. By contrast, on non-compact
manifolds, ellipticity is typically only a necessary, but not
sufficient condition to ensure Fredholmness; stronger conditions on an
operator $P$ are required to obtain that $P$ is
Fredholm. For example, on an asymptotically commutative Lie manifold,
the Fredholm condition is still controlled by the invertibility of a
function, which this time is an extension of the principal
symbol, and hence is defined on an extension of the cosphere bundle. This
phenomenon is studied in Section \ref{sec1}, where an index theorem is
proved for such operators by reducing to the case of operators that
are multiplication at infinity (which was studied in \cite{carvalho}).
In particular, we obtain in that section an index theorem for
asymptotically multiplication operators.

We shall assume throughout most of this paper that $M_0$ is a
non-compact Lie manifold with compactification $M$, althought some of
our results of the earlier sections may be true for more general
non-compact manifolds. For instance, the index theorem
of \cite{carvalho} is valid without any assumption on $M_0$.

\subsection*{Acknowledgements} We thank Max Planck Institute for
Mathematics for support while parts of this work were being completed.
We also thank Bernd Ammann and Ulrich Bunke for useful discussions.

\section{Asymptotically multiplication operators}\label{sec1}

In this section, we review some basic concepts and results to be used
in what follows, leading to the Atiyah-Singer index theorem in the
setting of \emph{non-compact} manifolds and operators that are
multiplication outside a compact set (or asymptotically so). Here, we keep
the manifolds quite general, while we consider a class of operators
that inherits naturally the properties of the compact manifold case.
For simplicity, we assume  that $M_0$ is endowed with a metric 
$g$ and that, as a topological space, it is the interior of a compact
manifolds with corners $M$ such that $TM$ restricts to $TM_0$ on $M_0$. We let $n$ be the dimension on $M_0$, which in this and the following section may be arbitrary, but in the last section will be assumed to be even.

\subsection{General calculus} \label{ss pso ac}
We consider for now a smooth manifold $M_0$ without boundary, not
necessarily compact, and a smooth vector bundle $E$ over $M_0$ that is
trivial outside a compact set in $M_0$. We denote by $d\vol_g$ the
volume form on $M_0$ defined by the metric. We also assume that $E$ is
endowed with a Hermitian metric, which is the trivial (product) metric
close to infinity. Typically, $M_0$ will coincide with the interior of
a given compact manifold with corners $M$.

We first make a short review of main results of the theory of
pseudodifferential operators on $M_0$ that we need in this paper, in
the setting of operators that are multiplication by a smooth function
outside a compact set.  We follow closely the approach
of \cite{carvalho} (see however also \cite{Horm} or \cite{Tay2} for
general references).

Let us recall that a smooth function $p: W\times\RR^n\into\CC^{N\times
  N}$ defines a symbol in the class $S^m(W\times\RR^n)$ of
\textit{symbols of order} $m$ if, and only if, for any compact set
$K\subset W$ and multi-indices $\alpha$, $\beta$, there exists
$\maC_{K,\alpha,\beta} > 0$ such that $
|\partial_x^\alpha \partial_\xi^\beta
p(x,\xi)|\leq \maC_{K,\alpha,\beta}(1+|\xi|)^{m-|\beta|}\, , $ for all
$x\in K$ and $\xi\in\RR^n$. An operator $P:
\maC^\infty_c(M_0;E) \to \maC^\infty(M_0;E)$ is said to be in the
class $\Psi^m(M_0;E)$ of \textit{pseudodifferential operators of
  order} $m$ on $M_0$ if, for any coordinate chart $W$ of $M_0$
trivializing $E$ and for any $h\in \maC^\infty_c(W)$, $hPh:
\maC_c^\infty(W)^N\to \maC_c^\infty(W)^N$ is a matrix of \psos\ of
order $m$ on $W$, that is, $hP(hu) = p(x, D)u$, with
\begin{equation}\label{pdoRn}
  h P(hu)(x) = p(x,D)u(x) :=
  (2\pi)^{-n} \int_{\RR^n}{p(x,\xi)\hat{u}(\xi)e^{ix\cdot \xi} d\xi},
\end{equation}
where $\hat{u}$ denotes the Fourier transform of $u$ and $p \in
S^m(W\times\RR^n)$.

We shall work only with classical symbols, that is, symbols that have
asymptotic expansions $p\sim\sum {p_{m-k}}$ with $p_{m-k}\in
S^{m-k}(W\times\RR^n)$ positively homogeneous of degree $m-k$ in
$\xi$. Let us denote by $\pi: T^*M_0 \to M_0$ the cotangent bundle of
$M_0$. The leading term $p_m$ in the expansion of $p(x,\xi)$ as a
classical symbol defines a smooth section of the bundle $\End(E)$ over
the cotangent bundle $T^*M_0$, the \textit{principal symbol} of $P$,
which is a smooth bundle homomorphism $\sigma_m(P) : \pi^*E \to
\pi^*E$, positively homogeneous on the fibers of $T^*M_0$. By choosing
a metric on $TM_0$, the class of principal symbols can be identified
with $\maC^\infty(S^*M_0;\End(E))$ where $S^*M_0$ is the unit sphere
bundle of the cotangent bundle.  An operator $P$ is said to
be \textit{elliptic} if $\sigma_m(P)$ is invertible on $S^*M_0$. We
shall regard the cosphere bundle $S^*M_0$ as the boundary of $T^*M_0$
using a radial compactification of each fiber. It is in this sense
that we shall often extend the principal symbol of an order zero
pseudodifferential operator to $T^*M_0$.

Under certain assumptions that will be satisfied in our setting, we
have that if $P\in \Psi^m(M_0;E)$, then $P_0:=
P(1+P^*P)^{-1/2}\in \Psi^0(M_0;E)$, with $P^*$ the formal adjoint, and
$P$ is Fredholm, respectively, elliptic, if, and only if, $P_0$ is,
with $\ind(P)=\ind(P_0)$.  Moreover, $\sigma_m(P)$ is homotopic to
$\sigma_{0}(P_0)$, as sections of $S^*M_0$ (the unit sphere bundle of
$T^*M_0$). Hence, for the purposes of index theory, we will mainly be
concerned with operators of order $0$.

We start with considering the class of pseudodifferential operators
that are \textit{multiplication outside a compact}, defined as
\begin{multline}\label{def psi mult P=P_1+p}
  \Psi^{{0}}_{mult}(M_0;E):=\{P = P_1 + p, \
    P_1 \in \Psi^0(M_0;E) \mbox{ has a compactly supported} \\
        \mbox{distribution kernel and } p \in\End(E) \mbox{ is
    bounded}\, \}.
\end{multline}
(For $m<0$, we consider $ \Psi^{{m}}_{mult}(M_0;E) :=
\Psi^{{0}}_{mult}(M_0;E) \cap  \Psi^{{m}}(M_0;E)$).
We have that any operator in $ \Psi^{{0}}_{mult}(M_0;E)$ is properly
supported and that $\Psi^0_{mult}(M_0;E)$ is a $*$-algebra
(see \cite{carvalho} for details).  Moreover, denoting by
$S_{mult}^0(T^*M_0;E)$ the set of bounded symbols in $S^0(T^*M_0;E)$
that are constant on the fibres of $T^*M_0 \to M_0$ outside a compact
of $M_0$, we have that there is a well-defined symbol map
\begin{equation}\label{symb map mpsos}
  \sigma_{0}: \Psi^0_{mult}(M_0;E)\to
  S^0_{mult}(T^*M_0;E)/S^{-1}_{mult}(T^*M_0;E),
\end{equation}
which is a surjective $*$--homomorphism with $\ker (\sigma_{0})
= \Psi^{-1}_{mult}(M_0;E) $, as it is the case if $M_0$ is compact.
Moreover, we have that $P\in \Psi^0_{mult}(M_0;E)$ is always bounded
as an operator on $L^2(M_0;E)$, and that $\Psi^{-1}_{mult}(M_0;E)$
consists of compact operators, again as in the classical case of
compact manifolds.

We endow the class of symbols $S^0_{mult}(T^*M_0;E)$ with the
$\sup$-norm, as a section of $\End(E)$ over $T^*M_0$.  Note that
$S^0_{mult}(T^*M_0;E)$ can be identified with the class of bounded
sections in $\maC^\infty(S^*M_0;\End(E))$ that are constant on the
fibers of $S^*M_0$ outside a compact $K\subset M_0$, and this is
consistent with regarding $S^*M_0$ as the boundary of (the radial
fibrewise compactification of) $T^*M_0$. (Recall that $E$ is
trivialized outside a compact set, so ``constant in a neighborhood of
infinity'' does indeed make sense.)

The class $\maC_{a}(S^*M_0;E)$ of {\em asymptotically multiplication
symbols} is defined as those functions $p = p(x,\xi)\in
\maC(S^*M_0;\Endo(E))$ such that $p(x,\xi)$ is bounded in the $\sup$-norm
and, for all $\epsilon>0$, there is a compact $K_{\epsilon}\subset
M_0$ such that, for all $x\notin K_{\epsilon}$,
\begin{equation}\label{C_a}
  \sup\limits_{\xi_1,\xi_2\in S^*M_0}
  \|p(x,\xi_1)-p(x,\xi_2)\|_{\End(E_x)}<\epsilon.
\end{equation}
Roughly speaking, the elements of $\maC_{a}(S^*M_0;E)$ are continuous
sections of $\Endo(E)$ over $S^*M_0$ that are bounded and
asymptotically independent of $\xi$ on the fibers of $S^*M_0$. It is
easily checked that it is a $C^*$-subalgebra of
$\maC_b(S^*M_0;\End(E))$, the class of continuous, bounded sections of
$\Endo(E)$. 

We now define the class of \textit{asymptotically
multiplication \psos} as
\begin{equation}\label{def clpsimult} 
        \psia(M_0;E)
        := \overline{ \psimult(M_0;E)}\subset \maB(L^2(M_0;E)),
\end{equation}
that is, the closure of $\psia(M_0;E)$ in the topology of bounded
operators on $L^2(M_0; E)$.

The point of the following lemma is that, once we consider
completions, we will need to replace operators that are multiplication
at infinity with asymptotically multiplication operators.

\begin{lem}\label{lemma.princ.symb} 
The principal symbol defines a natural map
$\Psi^0_{mult}(M_0;E) \to \maC_{a}(S^*M_0;E)$, which extends by
continuity to a surjective map $\psia(M_0;E) \to \maC_{a}(S^*M_0;E)$.
\end{lem}

\begin{proof}
We show that $\maC_{a}(S^*M_0;E)$ coincides with the closure of
\begin{multline}\label{Smult}
  \Xi := \{\,
  p \in \maC^\infty_b(S^*M_0;\End(E)), \, \text{ there exists }
  K \subset M_0 \text{ compact} \\ \text{ such that } p(x,\xi) \text{
  is independent of } \xi \text{ if } x \notin K \, \},
\end{multline}
and the result then follows as in the compact case.  It is readily
checked that any $p$ in the closure of \eqref{Smult} is
asymptotically multiplication. For the converse, let
$p\in \maC_a(S^*M_0;\End(E))$ and take
$\tilde{p}\in \maC_b^\infty(S^*M_0;\End(E))$ such that
$\|\tilde{p}-p\|_{\sup}<\epsilon$.  Let $K\subset M_0$ be compact such
that $\|p(x,\xi_1)-p(x,\xi_2)\|_{\End(E_x)}<\epsilon,$ for all
$x\notin K$, $\xi_1,\xi_2\in S^*M_0$, and let
$\phi\in \maC^\infty_c(S^*M_0)$ be such that $\supp\phi\subset M_0-K$,
$0\leq \phi\leq 1$, and $\phi(x,\xi) =1$, for $x\notin K^\prime$, with
$K^\prime$ compact such that $K\subset int(K^\prime)$.  Define
\begin{equation*}q(x,\xi):=(1-\phi(x,\xi))\;\tilde{p}(x,\xi) +
\phi(x,\xi)\;\tilde{p}(s(x)),\end{equation*} where $s$ is a fixed smooth section of
$S^*M_0$ (which exists since every connected non-compact manifold has
a nowhere vanishing vector field).  Then $q(x,\xi)\in
\maC_b^\infty(S^*M_0;\End(E))$ and for $x\notin K^\prime$,
$q(x,\xi)=\tilde{p}(s(x))$ is independent of $\xi$. Moreover,
\begin{multline*}
  \|q-p\|_{\sup} \leq \underset{{\xi\in S^*M_0}}{\sup\limits_{x\in
      K^\prime}}\|\tilde{p}(x,\xi)-p(x,\xi)\|_{\End(E_x)}
      + \underset{{\xi\in S^*M_0}}{\sup\limits_{ x\notin
      K,}} \|\tilde{p}(s(x))-p(x,\xi)\|_{\End(E_x)}\\
 \leq 2\|\tilde{p}-p\|_{\sup} + \underset{{\xi\in
      S^*M_0}}{\sup\limits_{ x\notin
      K,}} \|p(s(x))-p(x,\xi)\|_{\End(E_x)} \leq 3 \epsilon.
\end{multline*}
Hence, $p$ lies in the closure of $\Xi$ defined in
Equation \eqref{Smult}, and that concludes our proof.
\end{proof}

The following result can be proved much as in the compact case.

\begin{prop}\label{prop.symbol.ex.seq}
The principal symbol map (\ref{symb map mpsos}) is continuous and the
following sequence of $C^*$-algebras is exact
\begin{equation}\label{ex seq}
\begin{CD}
        0 @>>> \maK(M_0;E) @>>> \psia(M_0;E)
        @>{\sigma_{0}}>>\maC_{a}(S^*M_0;E) @>>> 0.
\end{CD}
\end{equation}
where now $\sigma_{0}: \psia(M_0;E) \to \maC_{a}(S^*M_0;E)$ denotes
the extension by continuity of the classical principal symbol map
$\sigma_{0} : \Psi^0_{mult}(M_0;E) \to \maC_{a}(S^*M_0;E)$.
\end{prop}

\begin{proof} 
The exactness at $\maC_{a}(S^*M_0;E)$ follows from Lemma
\ref{lemma.princ.symb}. Using a partition of unity and the fact
that our result is true in the compact case, we see that
$\Psi^{-1}_{mult}(M_0;E) \subset \maK$ as a dense subset. This proves
the exactness at {$\maK(M_0; E)$} and the fact that $\maK(M_0; E)$ is
contained in the kernel of $\sigma_{0}$.

As in the classical case of compact manifolds, the difficult case is
to prove that if an operator $T \in \psia(M_0;E)$ is in the kernel of
$\sigma_{0}$, then it is compact. Let then
\begin{equation*}
        T_n \in \psimult(M_0;E), \quad T_n \to T \ \ \text{ and
        } \ \ \sigma_{0}(T_n) \to 0.
\end{equation*}
Then we can replace the sequence $T_n$ with a sequence of operators
that are {\em zero} in a neighborhood of infinity.  Also, let
$\psi \in \CIc(M_0)$ have the support in a local coordinate
chart. Then $\psi T_n \psi \to \psi T \psi$ and $\sigma_{0}(\psi
T_n \psi) \to 0$. Using the case of a compact manifold, we see that
$\psi T \psi$ is a compact operator. From this we infer that $\psi_1
T \psi_2$ is also compact for {\em any} compactly functions $\psi_1$
and $\psi_2$. (One way to prove this is to consider first the case
when $\psi_1$ and $\psi_2$ have disjoint supports). Let
$0 \le \ldots \le \psi_k \le
\psi_{k+1} \le \ldots \le 1$ be an increasing sequence of compactly 
functions such that $\psi_n(x) \to 1$ for all $x$. (We are assuming
here that $M_0$ is $\sigma$-compact, which is always the case if $M_0$
has a compactification.)

We claim that $\psi_k T \psi_k \to T$. Since $\psi_k T \psi_k$ is
compact for any $k$, it will follow that $T$ is also compact. To prove
our claim, let $\epsilon > 0$ and choose $n$ such that $\|T - T_n\|
< \epsilon/3$. Then we can find $k_0$ such that $\|\psi_k T_n \psi_k -
T_n\|
\le \epsilon/3$ for $k \ge k_0$ since $T_n$ is assumed to be zero outside a compact set. Then 
\begin{equation*}
        \|T - \psi_k T \psi_k\| \le \|T - T_n\| + \| T_n - \psi_k
        T_n \psi_k\| + \| \psi_k (T_n - T) \psi_k\| \le \epsilon
\end{equation*}
for $k \ge k_0$. This completes our proof.
\end{proof}

It follows from Proposition \ref{prop.symbol.ex.seq} that $P\in
\psia(M_0;E)$ is a Fredholm operator if, and only if, its full
symbol is invertible in $\maC_a(S^*M_0;E)$ or, equivalently, in
  $\maC_b(S^*M_0;E)$.  See also \cite{Georgescu} for a discussion of
  Fredholm operators on non-compact manifolds.


\subsection{The Atiyah-Singer index theorem}

We now review the Atiyah-Singer index formula, applied to
asymptotically multiplication operators. (See for instance \cite{a-sI,
LawsonBook} for the details on the constructions below).

First, we define operators acting between sections of two different
vector bundles.  Let $E, F$ be vector bundles over $M_0$, with $E\cong
F$ outside a compact. We define $\Psi_{mult}^0(M_0;E,F)$ as the
subclass of $\Psi_{mult}^0(M_0;E\oplus F)$ of those operators that
induce $P: \maC^\infty_c(M_0;E)\to \maC^\infty(M_0;F)$.  We have so
also $S^0_{mult}(T^*M_0;E,F)\subset S^0_{mult}(T^*M_0;E\oplus F)$ and
all the results above hold, except that if $P\in
\Psi_{mult}^0(M_0;E,F)$ then $P^*\in \Psi_{mult}^0(M_0;F,E)$, so we
leave the setting of $C^*$-algebras.  In any case, an analogue of the
exact sequence given in Proposition \ref{prop.symbol.ex.seq} holds.

We now associate a $K$-theory class to an elliptic 
Fredholm operator in $\Psi_{a}^0(M_0;E,F)$.

\begin{lem}\label{lemma.def.ps1} 
For any elliptic, bounded $Q\in \Psi^{0}_{mult}(M_0;E,F)$ such that $q
:= \sigma_{0}(Q)$, here is a natural class $[\sigma_{0}(Q)]:=
[\pi^*E,\pi^*F, q]$ in the compactly supported $K$-theory of $T^*M_0$
obtained by extending $q$ to an invertible map outside a compact set
of $TM_0$ that is constant along the fibers of $TM_0 \to M_0$ outside
a compact set. This $K$-theory class is such that the Fredholm index
of $Q$ depends only on $[\sigma_{0}(Q)]$.
\end{lem}

\begin{proof} 
In order to associate a $K$-theory class to a Fredholm operator in
$\Psi_{a}^0(M_0;E,F)$, we start with noting that if
$Q \in \Psi^{0}_{mult}(M_0;E,F)$ is such that $q := \sigma_{0}(Q)$ is
invertible, then $q$ defines an isomorphism outside a compact subset
of $T^*M_0$ by homogeneity and the fact that it is constant on the
fibres outside a compact in $M_0$. Hence, regarding $q(x,\xi)$ as a
bundle map $\pi^*E\to \pi^*F$, $\pi:T^*M_0\to M_0$, we obtain the
desired definiton of $[\sigma_{0}(Q)]:= [\pi^*E,\pi^*F, q]$ as
in \cite{AtiyahBook, KaroubiBook}. The dependence of the index only on
$[\sigma_{0}(Q)]$ follows as in the classical case by noticing that
$Q$ is Fredholm as long as the principal symbol is invertible as
in \cite{carvalho}.
\end{proof}

Given now a Fredholm operator $P\in \psia(M_0;E,F)$ with invertible
symbol $\sigma_{0}(P)\in \maC_a(S^*M_0; E, F)$, by
Proposition \ref{prop.symbol.ex.seq}, we can take $q\in
S^0_{mult}(T^*M_0;E,F)$ sufficiently close to $\sigma_{0}(P)$ such
that $t\sigma_{0}(P)+(1-t)q$, $t\in[0,1]$, is an homotopy through
invertible symbols. We define the symbol class of $P$ as
\begin{equation}\label{symb class psia}
     [\sigma_{0}(P)] := [\pi^*E,\pi^*F, q]\in K^0(T^*M_0).
\end{equation}
This class is independent of $q$. If we take $q = \sigma_{0}(Q)$, with
$Q \in \Psi^{0}_{mult}(M_0;E,F)$, we have $\ind(P) = \ind(Q)$.
Moreover, if two Fredholm operators have the same symbol class, then
their indices coincide, and, in fact there is a well-defined
(analytic) index map
\begin{equation}
        \ind : K^0(TM_0) \to \ZZ, \quad
        [\sigma_{0}(P)] \mapsto \ind(P),
\end{equation}
where $P\in \psia(M_0;E,F)$ and we use the metric to identify
canonically $T^*M_0$ with $TM_0$. We summarize the above discussion in
the following lemma extending Lemma \ref{lemma.def.ps1}.

\begin{lem}\label{lemma.def.ps2} 
For any elliptic, bounded $Q\in \psia(M_0;E,F)$ such that $q
:= \sigma_{0}(Q)$, there is a natural class $[\sigma_{0}(Q)]:=
[\pi^*E,\pi^*F, q]$ in the compactly supported $K$-theory of $T^*M_0$
obtained by extending $q$ to an invertible map outside a compact set
of $TM_0$ that is asymptotically constant along the fibers of
$TM_0 \to M_0$. This $K$-theory class is such that the Fredholm index
of $Q$ depends only on $[\sigma_{0}(Q)]$.
\end{lem}

For a manifold $X$, we let $H^*(X)$, respectively $H_c^*(X)$, denote
the cohomology, respectively the compactly supported cohomology, of
$X$. Recall that throughout this section, we assume that $M_0$ is the
interior of a compact manifold with corners $M$.  Let $\overline {TM}$
be the radial compactification of the tangent bundle to $M$. Then the
pair $(\overline{TM} , \pa \overline{TM})$ is homeomorphic to the
similar pair associated to a manifold with boundary.  Hence the (even)
Chern character yields a map
\begin{equation*}
        \ch_0: K^0(TM_0)\to H_c^{2*}(TM_0) =
        H^{2*}(\overline{TM}, \pa \overline{TM}).
\end{equation*}
(We will also consider later on the odd Chern character $\ch_1$
defined on $K^1$.) Let also $Td(T_{\CC} M) \in H^*(M)$ denote the Todd
class of the complexified tangent bundle $TM \otimes \CC$. Note that
since $TM$ is oriented, there is a well-defined fundamental class
$[TM_0]\in H_{2n}(\overline{TM}, \pa \overline{TM})$ (see for
instance \cite{LawsonBook} for details on these constructions).
 
The following result is an immediate extension of a result in
\cite{carvalho} from operators that are multiplication outside a compact
to operators that are only asymptotically so. Let
$\pi: \overline{TM} \to M$ denote the natural projection.  We have
$ch_0[\sigma_{0}(P)] \in H^{2*}(\overline{TM}, \pa \overline{TM})$ and
$\pi^*Td({T_{\CC} M)} \in H^{2*}(\overline{TM})$ so their product is
in $H^{2*}(\overline{TM}, \pa \overline{TM})=H_c^{2*}(TM_0)$.

\begin{thm}\label{ind.thm.psia}
Let $P \in \psia(M_0;E,F)$ be such that $\sigma_{0}(P)$ is invertible
in $\maC_a(S^*M_0; E, F)$. Then $P$ is Fredholm and
\begin{equation*}
        \ind(P) = (-1)^n\ch_0[\sigma_{0}(P)]\pi^* Td(T_{\CC} M)[TM_0],
\end{equation*}
where $[\sigma_{0}(P)]$ is defined using Lemma \ref{lemma.def.ps2}.
\end{thm}

\begin{proof} 
The fact that $P$ is Fredholm follows from
Proposition \ref{prop.symbol.ex.seq}.  The rest of the proof follows
from the discussion before the statement of this theorem. Indeed, let
$P \in \psia(M_0;E,F)$ be elliptic. Then we can find
$P_0 \in \Psi^{0}_{mult}(M_0;E,F)$ that is close enough to $P$ such
that the straight line joining $P$ and $P_0$ consists of invertible
operators. Both the left hand side and the right hand side of our
index formula are homotopy invariant. For $P_0$ they are equal
by \cite{carvalho}. For $P$, they will be therefore equal as well, by
homotopy invariance.
\end{proof}

\subsection{Comparison spaces}\label{ss comp spaces}
In this subsection, we extend by deformation the index formula of
Theorem \ref{ind.thm.psia} to certain pseudodifferential operators on
noncompact manifolds that extend to the compactification $M$ of $M_0$
in a suitable sense. More precisely, we require the principal symbols
of our operators to extend to a so-called ``comparison space'' and
there is an invertible complete symbol at the boundary. We thus
generalize the approach in \cite{FoxHaskell1, MelroseScattering},
using homotopy to asymptotically multiplication symbols.

Recall that $M$ is a compactification of $M_0$ to a manifold with
corners. In this section, we fix a vector bundle $A$ over $M$ such
that $A \vert_{M_0}\cong TM_0$ (later, when we consider Lie
structures, such an $A$ will be naturally associated to $M_0$.)
Denote by $\overline{A}$ the fiber-wise radial compactification of
$A$, so $\overline{A}$ is a manifold with corners that fibers over $M$
with fibers closed balls of dimension $n$. We identify $A$ with $A^*$
using a fixed metric. Let $(S^*A)_{\partial M}$ be the restriction of
the cosphere bundle $S^*A$ to the boundary $\partial M$ of $M$. Define
\begin{equation}
\label{omega}
       \Omega := \partial(\overline{A}) =
       (S^*A) \cup \overline{A} \vert_{\partial M}
\end{equation} 
such that 
\begin{equation*}
        \maC(\Omega)
        = \{(f,g) \in \maC(S^*{A})\oplus \maC( \overline{A} \vert_{\partial
        M}) : f\vert_{(S^*{A})_{\partial M}} =
        g_{(S^*{A}) \vert_{\partial M}}\}.
\end{equation*}
The space $\Omega$ will play an important role in what follows. It is
closely related to a similar space introduced by Cordes and his
colaborators in his work on Gelfand theory for non-compact
manifolds \cite{CordesBook, Cordes}.

Let $\Psi_A(M_0;E)\subset \Psi^0(M_0;E)$ be a $*$-algebra of order
$0$, bounded, pseudodifferential operators. We say that $\Omega$ is
a \emph{comparison space} for $\Psi_A(M_0;E)$ if there is a surjective
homomorphism
\begin{equation}\label{complete symbol}
            \sigma_{full} : \overline{\Psi_A(M_0;E)} \to \maC(\Omega)
\end{equation}
such that $\sigma_{full}(P)\vert_{S^*M_0} = \sigma_{0}(P)$, with
kernel included in the algebra of compact operators. We call
$\sigma_{full}$ a {\em full} symbol and write $\sigma_{full} =
(\sigma_0, \sigma_\pa)$, where
\begin{equation*}
        \sigma_\pa
        : \overline{\Psi_A(M_0;E)} \to \maC(\overline{A}\vert_{\partial
        M})
\end{equation*}
is the \emph{boundary symbol morphism}. An operator with invertible
full symbol is called {\em fully elliptic}. We shall give an index
formula for fully elliptic operators in this setting, reducing to
asymptotically multiplication operators.

We see first that any function in $\maC(\Omega)$ can be homotoped over
the interior to an asymptotically multiplication symbol.  Since the
fibres of $\overline{A}$ are isomorphic to the $n$-dimensional half
sphere $\mathbb{S}^n_+$ and hence contractible, we have that $\Omega$
is homotopy equivalent to the space $\tilde{\Omega}$ obtained from the
cosphere bundle $S^*A$ by collapsing the fibers above points of the
boundary. More precisely, $\tilde{\Omega}:=(S^*{A})
\slash \sim$, with $(x,\xi)\sim(x,\xi^\prime)$, for $x\in \partial M$,
$\xi,\xi^\prime \in {S^*{A}_x}$. Since
\begin{multline*}
	\maC(\tilde{\Omega}) \cong \{f\in \maC(\Omega) : f \mbox{
  	constant on fibers over } \partial M\} \\
        \cong \{f \in \maC(S^*{A}) : f \mbox{ constant on fibers over
  	} \partial M\},
\end{multline*}
we conclude that every $f\in \maC(\Omega)$ is canonically homotopic to
some $\tilde{f}\in \maC(\Omega)$ constant on the fibres of
$\overline{A}\vert_{\partial M} \to \partial M$ (this is achieved by a
homotopy equivalence between $\Omega$ and $\tilde \Omega$).  Moreover,
if $f$ is invertible, the canonical homotopy between $f$ and
$\tilde{f}$ is through invertible functions.  We now have the
following:

\begin{lem}\label{const.fibres}
 If
$f\in \maC(\Omega,E)$ is constant on fibers of
$\overline{A}\vert_{\partial M} \to \partial M$, then $ f_0 :=
f\vert_{ S^*M_0} \in \maC_a(S^*M_0,E).$ In particular, let
$f\in \maC(\Omega,E)$, then $f$ is homotopic to
{
$\tilde{f} \in \maC(\tilde{\Omega}, E) \subset \maC(\Omega ,E)$ and hence
it satisfies 
} 
$f_0 := \tilde{f}\vert_{ S^*M_0} \in \maC_a(S^*M_0,E).$
\end{lem}

\begin{proof} 
We consider only the scalar case. Let $p\in \maC(\partial M)$ be such
that $f_{(S^*{A})_{\partial M}}=p$. It suffices to show that, given
$\epsilon >0$, there is a neighborhood $U$ of $\pa M$ in $M$ such that
for $y\in M_0\cap U$, $\xi, \xi^\prime \in S^*_yM_0$,
\begin{equation}\label{c}
        \|f_0(y,\xi)-f_0(y, \xi^\prime)\|<\epsilon
\end{equation}
so that (\ref{C_a}) follows with $K=M\setminus U$. We can do this
locally, so assume $U_x$ is a neighborhood of $x\in \pa M$ such that
$\pi^{-1}(U_x)\cong U_x\times S^{n-1}$, $\pi:S^*A\to M$. Let
\begin{equation*}
        V_x := \{(y,\xi)\in U_x\times
        S^{n-1} \colon \|f(y,\xi)-p(x)\|<\epsilon\slash 2\}.
\end{equation*}
By continuity of $f$ in $S^*A$, we have that $V_x$ is open in
$U_x\times S^{n-1}$. Moreover, since $f(y,\xi)=p(y)$, $y\in\pa M$, it
contains $W_x\times S^{n-1}$ for some $W_x\subset \pa M$ open. Hence,
$V_x$ can be taken as $\tilde{U}_x\times S^{n-1}$, for some open
neighborhood $\tilde{U}_x$ of $W_x$ in $M$ so that (\ref{c}) holds if
$y\in\tilde{U}_x$. The last statement follows from the fact that
$\Omega$ and $\tilde{\Omega}$ are homotopy equivalent.
\end{proof}

Let us notice that for $f \in \maC(\tilde \Omega, E)$, the restriction
$f\vert_{S^*M_0}$ completely determines $f$ since $S^*M_0$ is dense in
$\tilde \Omega$.  Let now $P\in \overline{ \Psi_A^0(M;E)} $ have
invertible symbol
\begin{equation*}
        \phi = (\sigma_{0}(P), \sigma_{\pa}(P)) \in \maC(\Omega,E).
\end{equation*}
(In the terminology introduced earlier, $P$ is fully elliptic.) From
the previous lemma, there is an invertible
{
$\tilde{\phi} \in \maC(\tilde{\Omega}, E) \subset \maC(\Omega, E)$ 
}
homotopic through invertibles to $\phi$, with $\tilde{\sigma}
:= \tilde{\phi}\vert_{ S^*M_0} \in \maC_a(S^*M_0,E)$ and
$\sigma:=\sigma_0(P)$ and $\tilde{\sigma}$ are homotopic (over
$M_0$). If we let $\tilde{\sigma} = \sigma_{0}(\tilde{P})$, for some
$\tilde{P} \in \Psi^0_a(M_0;E) \cap \overline{ \Psi_A^0(M;E)}$, then
$\tilde{P}$ is Fredholm, since $\tilde{\sigma}$ is
invertible. Moreover, from the surjectivity of the complete symbol map
(\ref{complete symbol}), there is a continuous family
$P_t\in \overline{ \Psi_A^0(M;E)}$, $t\in [0,1]$, lifting the homotopy
between $\phi$ and $\tilde \phi$. Hence $P_0=P$ and $P_1=\tilde{P}$
and $\ind(P) = \ind(\tilde{P})$.

We conclude also that any fully elliptic operator $P \in
\overline{ \Psi_A^0(M;E, F)}$ has associated a well-defined 
$K$-theory class $[\tilde{\sigma}_{full}(P)]$ extending by homotopy
the definition in Equation \eqref{symb class psia} as follows.  We
know that $P$ is homotopic to some $\tilde P \in \Psi^0_a(M_0;E)$
through Fredholm operators in $\Psi^0_A(M_0;E)$ and
$[\tilde{\sigma}_{full}(P)] := [\sigma_0(\tilde P)]$, so that
\begin{equation}\label{k0 class comp sp}
        [\tilde{\sigma}_{full}(P)] := [\sigma_0(\tilde{P})] =
        [\pi^*E, \pi^*F, \tilde p] \in K^0(TM_0),
\end{equation}
where we assume $\tilde p$ is an extension of the principal
symbol of $\tilde P$ to a function that 
is multiplication at infinity and homotopic to
$\sigma_0(P)$ over the interior. In particular.
\begin{equation*}
        \ind(P) = \ind(\tilde P) = \ind([\sigma_0(\tilde{P})])
        = \ind([\tilde{\sigma}_{full}(P)]).
\end{equation*}

Consider now, for $P \in \overline{ \Psi_A^0(M;E, F)}$ fully elliptic,
\begin{equation}\label{k1 class comp sp}
        [\sigma_{full}(P)] :=
        [(\sigma_0(P), \sigma_\pa(P))]\in K_1(\maC(\Omega))\cong
        K^1(\Omega)
\end{equation}
where $\sigma_\pa(P)\in \maC(\overline{A}_{\pa M})$ and
$\sigma_0(P)\in \maC(S^*A)$ denote the boundary and principal symbol,
respectively.  Let us consider the connecting map $\pa: K^1(\Omega)\to
K^0(TM_0)$ in the long exact sequence of the pair $(\overline{A}, \pa
\overline{A} )=(\overline{A}, \Omega)$. We summarize the above discussion
to the following generalization of Lemma \ref{lemma.def.ps2}.

\begin{lem}\label{lemma.def.ps3} 
For any fully elliptic $P \in \overline{ \Psi_A^0(M;E, F)}$ 
there is a natural class
\begin{equation*}
        [\tilde{\sigma}_{full}(P)] := [\pi^*E,\pi^*F, p]
\end{equation*}
in the compactly supported $K$-theory of $T^*M_0$ obtained by
extending $\sigma_{full}(P)$ to a continuous endomorphism $p$
invertible outside a compact set of $TM_0$. This $K$-theory class is
such that the Fredholm index of $P$ depends only on
$[\tilde{\sigma}_{full}(P)]$ and 
\begin{equation*}
        [\tilde{\sigma}_{full}(P)]
        = \pa[(\sigma_0(P), \sigma_\pa(P))]=\pa[\sigma_{full}(P)].
\end{equation*}
\end{lem}

Let us also notice that $\Omega$ is homotopically equivalent to the
boundary of an oriented smooth manifold with boundary, and hence it
has a well defined fundamental class $[\Omega] \in
H_{2n-1}(\Omega)$. If $[\overline{A}]$ denotes the fundamental class
of $\overline{A}$ in $H_{2n}(\overline{A}, \Omega)$, then $[\Omega] =
\pa [\overline{A}]$. Using the compatibility of the boundary maps in K-theory
and cohomology, that is, the fact that the Chern character is a
natural transformation of cohomology theories
(see \cite{nistorHighInd} for an extension of this result to
non-commutative algebras), we obtain the following result as a
consequence of the Atiyah-Singer index formula extended to operators
that are asymptotically multiplication operators
(Theorem \ref{ind.thm.psia}).

As before, let $Td(T_\CC M)$ denote the Todd class of the complexified
tangent bundle of $M$, and $\pi: \overline{T^*M}\to M$. Also, we denote by
$\pi_\Omega: \Omega \to M$ the natural projection.

\begin{thm}\label{ind thm comp spaces}
Let $\Omega$ be a comparison space for $\Psi_A^0(M;E, F)$ and
$P \in \overline{\Psi_A^0(M;E, F)}$ be fully elliptic operator
(that is, an elliptic operator with $\sigma_\pa(P)$ invertible in
$\maC(\overline{A}\vert_{\pa M})$).  Then $P$ is Fredholm and
\begin{equation*}
	\ind(P) = {(-1)^n}\ch_0[\tilde
	{\sigma}_{full}(P)] \pi^*Td(T_\CC M) [TM_0] =
	{(-1)^n}\ \ch_1[\sigma_{full}(P)] \pi_\Omega^*Td(T_\CC M)
	[\Omega],
\end{equation*}
where $[\tilde{\sigma}_{full}(P)]$ is defined using
Lemma \ref{lemma.def.ps3}.
\end{thm}

\begin{proof} Again, the proof of the first equality follows from the
discussion before the statement of the theorem. Indeed, let us choose
a homotopy between $P$ and $\tilde P \in \Psi_a^0(M_0; E, F)$ through
Fredholm operators $\overline{\Psi_A^0(M;E, F)}$. Both the left hand side
and the right hand side(s) of the index formula of this theorem are homotopy
invariant. For $\tilde P$ they are equal in view of Theorem \ref{ind.thm.psia}, by
homotopy invariance, they will be equal also for $P$.
To prove the last equality, we just use the fact that the
Chern character is compatible with the boundary maps in $K$-theory and cohomology. (A proof of a generalization of this result to non-commutative algebras can be found in 
\cite{nistorHighInd}.) 
\end{proof}

\section{Index formula on asymptotically commutative Lie manifolds}\label{sec2}
 
From now on,  we endow $M_0$ with the structure of a Lie manifold with compactification $M$ and
structural Lie algebra of vector fields $\maV$ (see below for the definitions). There is associated to
$(M, \maV)$ a well-behaved pseudodifferential calculus and, for
operators in this calculus, Fredholm criteria follow from the
pseudodifferential calculus of operators on
groupoids \cite{ALNpdo, LMN1}.

We show that if we introduce the additional assumptions on the
structural Lie algebra $\maV$ that it be {\em asymptotically
commutative}, then there will exist a (commutative) complete symbol
and hence we can apply the results in the previous section.
Recall that in this section $n$ may be arbitrary (in
the following section will be assumed to be even).

\subsection{Operators on Lie manifolds}\label{s lie mfld}
In this section, $M$ will denote a compact manifold with corners and
$M_0 = int(M)$, as before. Also, let ${\maV}_M$ denote the Lie algebra
of vector fields that are tangent to all faces of $M$.  We always
assume that each hyperface $H \subset M$ is an embedded submanifold of
$M$ and hence that it has a defining function $x_H$ (recall that this
means that $x_H$ is smooth on $M$, $x_H\geq 0$, $H=\{x_H=0\}$, and
$dx_H\neq 0$ on $H$).

We recall the main definitions of
\cite{ALNgeo,ALNpdo}.  We say that a Lie
subalgebra $\maV \subset {\maV_M}$ is a {\em structural Lie algebra of
  vector fields} if it is a Lie algebra with respect to the Lie
  bracket and it is also a finitely generated, projective,
  $\maC^\infty(M)$-module. By the Serre-Swan theorem, we have that
  there exists a vector bundle $A$ such that $\maV \cong
\Gamma( A)$. Moreover, there is a vector bundle morphism $\rho: A \to TM$,
called {\em anchor map}, which induces the inclusion map
$\rho: \maV = \Gamma(A) \to \Gamma(TM)$. It thus follows that $ A$ with
the given structure is naturally a Lie algebroid.

\begin{defn} \label{def.Lie.man}
A \textit{Lie manifold} $M_0$ is given by a pair $(M,\maV)$ where
$M_0=int(M)$ and $\maV$ is structural Lie algebra of vector fields
such that $\rho\vert_{M_0} : A\vert_{M_0}\to TM_0$ is an isomorphism.
\end{defn}

A metric on $M_0$ that is obtained from a metric on $A$ 
by restriction to $A\vert_{M_0} \cong TM_0$ will be called a {\em
compatible metric} on $M_0$. Any two such metrics are Lipschitz
equivalent.  We fix one of these metrics on $M_0$ in what follows.

To a Lie manifold $(M, \maV)$ we associate the algebra
$\Diff_{\maV}(M_0)$ of {\em $\maV$-differential operators on $M_0$},
defined as the enveloping algebra of $\maV$ (generated by $\maV$ and
$\maC^\infty(M)$). It was shown in \cite{ALNgeo} that
$\Diff_{\maV}(M_0)$ contains all geometric operators on $M_0$
associated to a compatible metric, such as the Dirac and generalized
Dirac operators. (This property of $\Diff_{\maV}(M_0)$ will be used in
Section \ref{sec3}.) One defines differential operators acting between
sections of vector bundles $E,F$ over $M$ as
\begin{equation*}
        \Diff_{\maV}(M_0;E,F):= e_FM_N(\Diff_{\maV}(M_0))e_E,
\end{equation*} 
where $e_E, e_F$ are projections onto $E, F\subset M\times \CC^N$.
In \cite{ALNpdo}, a class of pseudodifferential operators
associated to a given Lie structure at infinity is defined by a
process of microlocalizing $\Diff_{\maV}(M_0;E,F)$. We outline this
construction below. 

Recall that we first define the class $S^m(A^*)\subset C^\infty(A^*)$
as functions satisfying the usual symbol estimates on coordinate
patches trivializing $A^*$, which are moreover classical symbols. By
inverse Fourier transform on the fibres, each symbol $a\in S^m(A^*)$
defines a distribution $\maF_{fib}^{-1}(a)$ on $A$ that is conormal to
$M$. By restriction, $\maF_{fib}^{-1}(a)$ defines a distribution on
$TM_0$ conormal to $M_0$. We fix a metric on $A$ which then defines a
compatible metric. We denote by $\exp$ the (geodesic) exponential map
associated to this metric (yielding $\exp_x : T_x M_0 \to M_0$ for
each $x \in M_0$). Now for some $r>0$, let
\begin{equation*}
        \Phi : (TM_0)_r \to V_r \subset M_0 \times M_0, v \in
        (T_x M_0)_r \mapsto (x, exp_x(-v))
\end{equation*}
be the diffeomorphism given by the Riemann-Weyl fibration, where
$(TM_0)_r$ are the vectors with norm less than $r$, $V_r$ is an open
neighborhood of the diagonal $M_0 \cong \Delta_{M_0} \subset M_0^2$,
and $r>0$ is less than the injectivity radius of $M_0$, which is known
to be positive. Fix a smooth function $\chi$, with $\supp \chi\in A_r$
and $\chi=1$ on a neighborhhod of the zero section of $A$, which is
identified with $M$. For $a\in S^m(A^*)$, define a distribution on
$M_0^2$, conormal to $M_0$ by
\begin{equation}\label{ker.psi}
        q_\chi(a) := \Phi_*( \chi\maF^{-1}_{fib}(a)).
\end{equation}
Let $a_\chi(D)$ denote the operator on $M_0$ with Schwartz kernel
$q_\chi$. Then $a_\chi(D)$ is a properly supported (if $r<\infty$)
pseudodifferential operator on $M_0$.

For each $X\in \maV=\Gamma(A)$, let
$\psi_X: \maC^\infty_c(M)\to \maC^\infty_c(M)$ be the operator induced
by the global flow $\Psi_X: \RR \times M\to M$ by evaluation at $1$.

\begin{defn} \label{def.psos.Lie.man}
The space $\Psi_\maV^m(M_0)$ of \emph{pseudodifferential operators
generated by the Lie structure at infinity} $(M,\maV)$ is the linear
space of operators $\maC^\infty_c(M_0)\to \maC^\infty_c(M_0)$
generated by $a_\chi(D)$, with $a\in S^m(A^*)$ and
$b_\chi(D) \psi_{X_1}\cdots \psi_{X_k}$, with $b\in S^{-\infty}(A^*)$
and $X_1, \cdots , X_k\in\maV=\Gamma(A)$.
\end{defn}

We define similarly the space $ \Psi^m_{\maV}(M_0;E,F)$ of
pseudodifferential operators acting between sections of vector bundles
$E,F$ over $M$.

As for the usual algebras of pseudodifferential operators, we have the
following basic property of the principal symbol (Proposition
2.6 \cite{ALNpdo}): the principal symbol establishes isomorphisms
\begin{equation*}
        \sigma_m : \Psi^m_\maV(M_0)/\Psi^{m-1}_\maV(M_0)\to
        S^m(A^*)/S^{m-1}(A^*)\cong \maC^\infty(S^*A).
\end{equation*}

We note that the set of all $a_\chi(D)$, with $a\in S^\infty(A^*)$ is
{\em not} closed under composition of operators, that is why we
consider extra operators in $\Psi^{-\infty}(M_0)$.  To show that
$\Psi_\maV^m(M_0)$ is indeed closed under composition, results from
the following section are needed.



\subsection{Operators on groupoids}
In order to obtain algebraic properties and, in particular, Fredholm
criteria, an important result is that $ \Psi_\maV^m(M_0)$ can be
recovered from an algebra of pseudodifferential operators on a
suitable groupoid integrating $A$.  We review the main definitions of
the theory of pseudodifferential operators on groupoids, for the
benefit of the reader (see \cite{LMN, NWX}).

For a Lie groupoid $\maG$ with space of units given by a manifold with
corners $M$, with $d,r: \maG\to M$ the domain and range maps, $P=(P_x)
\in \Psi^m(\maG)$ is defined as a smooth family of pseudodifferential
operators on the fibres $\maG_x:=d^{-1}(x)$, $x\in M$, which is
right-invariant, that is, $U_gP_{d(g)}=P_{r(g)}U_g$ where $U_g:
C^\infty(\maG_{d(g}) \to C^\infty(\maG_{r(g)}), U_g(f)g^\prime:=
f(g^\prime g)$. Recall that the definition of a Lie groupoid requires
the sets $\maG_x := d^{-1}(x)$ to be smooth manifolds (no corners). We
also assume that this family is uniformly supported, in that
\begin{equation*}
        \supp(P) = \overline{\cup_x\mu(\supp(K_x))}\subset \maG
\end{equation*}
is compact, where $\mu(g,h)=gh^{-1}$ and $K_x$ denotes the Schwartz
kernel of $P_x$ (a distribution on $\maG_x\times \maG_x$).  In this
case, each $P_x$ is properly supported, so that the composition
$gh^{-1}$ is well defined. Moreover, $P$ acts on $C^\infty(\maG)$.
Let $T^d\maG=\ker d^*=\cup T_x\maG_x$ be the $d$-vertical tangent
bundle and denote by $A(\maG):=\left(T^d\maG\right)_M$ the {\em Lie
algebroid} of $\maG$ and $S^*A(\maG):=(A(\maG)^* \setminus
0) \slash\RR^*_+$ its cosphere bundle.  Let us fix a metric on
$A$. This choice defines a principal symbol map
$\sigma_m: \Psi^m(\maG) \to \CIc(S^*A(\maG))$, which is surjective,
with kernel $\Psi^{m-1}(\maG)$. One can define similarly operators
acting between sections of vector bundles: if $E$ is a vector bundle
over the space of units $M$, then $\Psi^m(\maG, r^*E)$ is well-defined
as above.

For each $x\in M$, we consider the \emph{regular representation}
$\pi_x$ of $\Psi^\infty(\maG) $ on ${\maC^\infty(\maG_x)}$ defined as
$\pi_x(P):=P_x$. When restricted to order zero operators, this is a
bounded $*$-representation, for all $x$.  For $P\in \Psi^0(\maG)$, let
\begin{equation}\label{eq.def.reduced}
         \|P\|_ r := \sup_{x\in M} \|\pi_x(P)\| 
\end{equation} 
be the \emph{reduced $C^*$-norm}. We shall also need the \emph{full
$C^*$-norm} defined as
\begin{equation*}
        \|P\|:=\sup_\rho \|\rho(P)\|,
\end{equation*}
where $\rho$ ranges through bounded $*$-representation of
$\Psi^0(\maG)$ such that for $T\in \Psi^{-\infty}(\maG)$,
$\rho(T)\leq \|T\|_1$, with $\| \; \|_1$ defined by integrating the
Schwartz kernels over the fibres (see \cite{LMN} for the precise
definitions). Endowing $ \Psi^0(\maG)$ with the full norm $\| \; \|$,
we have that the principal symbol extends to a bounded
$*$-homomorphism
\begin{equation}\label{eq.def.princ.symb}
        \sigma_0
        : \overline{ \Psi^0(\maG)} \to \maC_0(S^*A(\maG)),
\end{equation}
surjective, with kernel $ C^*(\maG) :
= \overline{\Psi^{-\infty}(\maG))}$. (A similar result holds for the
reduced norm.)

If $Y\subset M$ is an invariant subset (that is,
$d^{-1}(Y)=r^{-1}(Y)$), then $\maG_Y:=d^{-1}(Y)$ is also a continuous
family groupoid, with units $Y$ and there is a well-defined
restriction map $\maR_Y: \Psi^{m}(\maG;E)\to \Psi^{m}(\maG_Y; E_Y)$.
In this case, Lemma 3 in \cite{LMN} gives that the following sequence
is exact:
\begin{equation}\label{ex seq lie}
  0 \longrightarrow C^*(\maG_{M \setminus
  Y}) \longrightarrow \overline{\Psi^0(\maG)} \stackrel{(\sigma_0, \maR_Y)}{-\!\!\!
  -\!\!\! -\!\!\!
  -\!\!\! \longrightarrow} \maC(S^*A(\maG)) \times_{\maC(S^*A(\maG)_{Y})} 
  \overline{\Psi^0(\maG_Y)} \longrightarrow
  0,
\end{equation}
where the fibered product $\overline{\Psi^0(\maG_Y)}
\times_{\maC_0(S^*A(\maG)_{Y})} \maC_0(S^*A(\maG))$ is
defined as the algebra of pairs $(Q,
f) \in \overline{\Psi^0(\maG_Y)} \times \maC_0(S^*A(\maG))$ such that
$\sigma_0(Q)=f_{\vert{S^*A(\maG)_{Y}}}$.

If $M_0=int (M )$ is an invariant subset, one can define the
so-called \textit{vector representation} $\pi_{M_0}$, which associates
to $P\in \Psi^0(\maG)$ a pseudodifferential operator $\pi_{M_0}(P):
\maC^\infty_c(M_0)\to \maC^\infty_c(M_0)$ by the formula
$\pi_{M_0}(P)u = u_0$, with $P(u \circ r) = u_0 \circ r$
\cite{LN}. Recall that a Lie groupoid is called {\em $d$-connected}
if all the sets $\maG_x := d^{-1}(x)$ are connected. If $A \to M$ is a
Lie algebroid on $M$, we say that $\maG$ integrates $A$ if $A(\maG) =
A$. We shall need Theorem 3.3 from \cite{ALNpdo}, which gives that

\begin{thm}
\label{psiV as psiG}
Let $(M,\maV)$ be a Lie manifold with Lie algebroid $A$ and $\maG$ be
a $d$-connected groupoid over $M$ integrating $A$. Then
$ \Psi_\maV^m(M_0)\cong \pi_{M_0}(\Psi^m(\maG)).$
\end{thm}

The right-hand-side is well defined since, as we shall see next, one
can always assume that $M_0$ is an invariant subset of such $\maG$. In
particular, it follows that the classes $\Psi_\maV^m(M_0)$ define a
filtered algebra on $\Psi_\maV^\infty(M_0)$.

The problem of integrating Lie algebroids was solved in
\cite{CrainicFernandes}, though for our purposes, the results in
\cite{Nistor} suffice. Namely, $M_0$ and $\pa M$ form
an $A$-invariant stratification of $M$, so it follows from the glueing
theorem in \cite{Nistor} that it suffices to integrate along these
strata. Since the anchor map $\rho$ is a diffeomorphism over the
interior, we can take the $d$-connected groupoid $\maG$ integrating
$A$ to coincide with the pair groupoid over the interior, meaning that
$\maG_{M_0}\cong M_0\times M_0$, in case $M_0$ is connected. (The
general case of non-connected $M_0$ can be reduced to the connected
case by taking the compactification of each connected component.) It
then follows from \cite{Nistor} that, if $\maG_{\partial M} $ is a
groupoid integrating $A_{\pa M}$, then
\begin{equation}\label{G integ V}
  \maG= \maG_{M_0}\sqcup \maG_{\partial M} \cong (M_0\times
  M_0) \cup \maG_{\partial M} ,
\end{equation}
has the structure of a differentiable groupoid with Lie algebroid $A$.
We see that $M_0$ is indeed an invariant subset, and moreover, since
$\maG_{M_0}$ is the pair groupoid, one has that $C^*(\maG_{M_0})\cong
\maK(L^2(M_0))$, the isomorphism being induced either by the vector 
representation $\pi_{M_0}$ or by $\pi_x$, $x\in M_0$, noting that
these representations are equivalent through the isometry
$r: \maG_x\to M_0$. In particular, the vector representation
$\pi_{M_0}$ is bounded. Fredholm criteria now follow from the exact
sequence (\ref{ex seq lie}) as in \cite{LMN} (Theorem 4).

{\em From now on we shall assume that $\maG$ is a $d$-connected Lie
groupoid integrating the Lie algebroid $A \to M$ defined by a Lie
manifold $(M, \maV)$.  We shall also assume that the vector
representation $\pi_{M_0}$ is injective on $\overline{\Psi^0(\maG)}$.}
In particular, $\Psi^0(\maG) \cong \Psi^0_{\maV}(M_0)$. Moreover,
since $C^*(\maG) \subset \overline{\Psi^0(\maG)}$ and $\pi_{M_0}$
factors through the reduced $C^*$-algebra of $\maG$, so that
$\pi_{M_0}(C^*_r(\maG))=\pi_{M_0}(C^*(\maG))$, we hence obtain that
$\maG$ is amenable, in that the reduced and full norms coincide.

We shall use the isomorphism above to carry to $\Psi^m_{\maV}(M_0)$
all concepts defined for $\Psi^m(\maG)$. At the level of symbols, we
have $\sigma_m(P)=\sigma_m(\pi_{M_0}(Q))=\sigma_m(Q)$ on $M_0$, for
any $P \in \Psi^m_{\maV}(M_0)$. We shall also need the map of
restriction to the boundary for operators on $(M,\maV)$
\begin{equation*}
        \sigma_\pa
        : \overline{\Psi^0_{\maV}(M_0)} \to \overline{\Psi^0
        (\maG_{\pa M})}, \quad P\mapsto \maR_{\pa}(Q) = Q\vert_{\pa
        M},
\end{equation*}
where $\pi_{M_0}(Q)=P$ and $\maR_\pa: \Psi^0(\maG)\to \Psi^0(\maG_{\pa
M})$ is restriction to the boundary.

\begin{prop}\label{prop.fredholm.lie}
Let $(M,\maV)$ be a Lie manifold with Lie algebroid $A$ and $\maG$ be
a $d$-connected groupoid as in \eqref{G integ V} satisfying
$A(\maG) \simeq A$. Assume that the representation $\pi_{M_0}$ is
injective on $\overline{\Psi^0(\maG)}$, as above. Then
\begin{multline*}
        \overline{\Psi^0(\maG)}/\maK \cong \maC(S^*A) \times_{\maC(S^*A_{\pa
     M})} \overline{\Psi^0(\maG_{\pa M})}\\ 
        := \{ (a,
     Q) \in \maC_0(S^*A) \times \overline{\Psi^0(\maG_{\pa M})} ,
     a \vert_{\pa M} = \sigma_0(Q) \in \maC(S^*A_{\pa M})\, \}
\end{multline*} 
and $P \in \overline{\Psi_\maV^0(M_0)}$ is Fredholm if, and only if,
it is elliptic and $\sigma_\pa(P)$ is invertible in
$\overline{\Psi^0(\maG_{\pa M})} $.
\end{prop}

\begin{proof}
Since $\pi_{M_0}$ is injective and $\pi_{M_0}C^*(\maG_{M_0})\cong
\maK(L^2(M_0))$, we have the induced representation 
$\pi': \overline{\Psi^0(\maG)} /C^*(\maG_{M_0})\to
{\maB(L^2(M_0))}/\maK$, which is also injective. Hence,
$P=\pi_{M_0}(Q)$ is Fredholm if, and only if the class of $Q$ is
invertible in $\overline{\Psi^0(\maG)} /C^*(\maG_{M_0})$. The result
follows from (\ref{ex seq lie}).
\end{proof}

Moreover, the amenability of $\maG$ yields that the restriction
$\maG_{\pa M}$ is also amenable \cite{Ren} Prop. 3.7). In this case,
$\rho:= \Pi_{x\in \pa M}\pi_x$ is an injective representation of
$\Psi^0(\maG_{\pa M})$ and $\sigma_\pa(P)$, as above, is invertible
if, and only if, $\sigma_\pa(P)_x = Q_x$ is invertible for all
$x\in \pa M$, with $\pi_{M_0}(Q)=P$. (The same is true also for
$\Psi^\infty(\maG_{\pa M})$, since if $\rho(P) = 0$, then
$\rho(P(1+P^*P)^{-1/2}) = 0$.)

Elliptic operators $P$ with invertible $\sigma_{\pa}(P)$ are sometimes
called {\em fully elliptic} and the algebra $\Psi^0(\maG_{\partial
M})$ is the so-called {\em indicial algebra}. If $\pi_{M_0}$ is not
injective for some $x\in M_0$, then we only have a sufficient
condition for Fredholmness.

To finish this section, we prove a result that will later enable us to
compute the index of operators with order $m>0$ from the index of
order $0$ operators.

\begin{lem}\label{lemma.order.red}
Let $Q \in \Psi^m_{\maV}(M_0;E)$ and $P := Q(1+Q^*Q)^{-1/2}$. Then
$P \in \overline{\Psi^0_{\maV}(M_0; E)}$
\end{lem}

\begin{proof}
Let $\maG$ be the canonical groupoid integrating $(M,  \maV 
)$. It follows from groupoid calculus applied to
$\Psi^0(\maG)$, more precisely from Theorem
7.2 in \cite{LN}, that if $L \in {\Psi^{2m}(\maG)}$ is such that
$L \geq 1$ and $\sigma_{2m}(L) > 0$ then $SL^{-1/2} \in
\overline{\Psi^0(\maG)}$, for any $S\in \Psi^m(\maG)$.
{F}rom Theorem \ref{psiV as psiG}, let $R \in {\Psi^m(\maG)}$ be such
that $\pi_{M_0}(R) = Q$, with $\pi_{M_0}$ the vector
representation. Then
\begin{equation*}
        1 + Q^*Q = 1 + \pi_{M_0}(R)^*\pi_{M_0}(R) = \pi_{M_0}(1+R^*R)
\end{equation*}
and we can apply the result above to $1+R^*R$ to obtain $
R(1+R^*R)^{-1/2}\in \overline{{ \Psi}^0(\maG)}$. Hence
$\pi_{M_0}(R(1+R^*R)^{-1/2}) \in \overline{{ \Psi_\maV}^0(M_0)}$. It
follows from the definitions that
\begin{equation*}
	\pi_{M_0}((1+R^*R)^{-1/2}) =
	\pi_{M_0}(1+R^*R)^{-1/2},
\end{equation*} 
so that $P = Q(1+Q^*Q)^{-1/2} \in \overline{\Psi^0_{\maV}(M_0)}$.
\end{proof}
 
Note that if $P$ and $Q$ are as in the lemma above
(Lemma \ref{lemma.order.red}), $\sigma_m(Q)$ and $\sigma_0(P)$ are
homotopic as sections of $S^*A$. Moreover, if we define the map of
restriction to the boundary $\sigma_\pa
: \Psi^m_\maV(M_0;E)\to \Psi^m(\maG_{\pa M}; r^*E)$ given, as before,
by $\sigma_\pa(Q):=\maR_\pa(R)$, with $\pi_{M_0}(R)=Q$, then it
follows from the proof that
\begin{equation*}
        \sigma_\pa(P) = \sigma_\pa(Q(1+Q^*Q)^{-1/2}) = \sigma_\pa(Q)
        (1 + \sigma_\pa(Q)^*\sigma_\pa(Q))^{-1/2},
\end{equation*}
hence $\sigma_\pa(P)$ is invertible if, and only if, $\sigma_\pa(Q)$
is.  We say that $Q\in \Psi^m_{\maV}(M_0;E)$ is {\em fully elliptic}
if, and only if, $P$ is. In that case, $P$ is Fredholm and $Q$ will
also be Fredholm, in the setting of unbounded operators, with
$\ind(Q)=\ind(P)$ (see Section \ref{ssec.Callias}).


\subsection{The asymptotically commutative case\label{ss ind form lie}}

In this subsection, we prove an index formula for certain classes of
pseudodifferential operators on Lie manifolds whose associated
groupoids are such that the restrictions at the boundary yield bundles
of commutative Lie groups. The main point is giving conditions that
yield commutativity of the algebra $\Psi^0(\maG_{\partial M})$, using
the notation of the previous section, so that Fredholmness depends on
invertibility in an algebra of functions, thus reducing to the setting
considered in Section \ref{ss comp spaces}.  This is known to hold for
the scattering and double-edge calculus \cite{LMN, LMor2, LN,
MelroseScattering, parenti}.  The dimension of $M$ is denoted by $n$,
as before. Recall that in this section, we do not assume $n$ to be
even.

\begin{defn}\label{def.com.inf}
Let $(M, \maW)$ be a connected Lie manifold with Lie algebroid $\pi :
A_{\maW} \to M$ with the property that any $X \in \maW$ vanishes at
the boundary $\pa M$ (that is, on any face of the boundary) and the
resulting Lie algebras $A_{\maW, x} := \pi^{-1}(x)$ are commutative. A
Lie manifold $(M, \maW)$ with this property will be called an {\em
asymptotically commutative} Lie manifold, and $\maW$ will be called
{\em commutative at infinity}.
\end{defn}

(We reserve the notation $\maW$ for asymptotically
commutative structural Lie algebras of vector fields, whereas $\maV$
will denote a general such structural Lie algebra of vector fields.)

Let $(M, \maW)$ be an asymptotically commutative Lie manifold. Then a
groupoid integrating $A_{\maW}\vert_{\pa M}$ is $A_{\maW}\vert_{\pa
M}$ itself (since the commutative Lie algebra $\RR^n$ identifies with
itself via the exponential map). According to \cite{Nistor}, there
will be a unique Lie manifold structure on the disjoint union
\begin{equation}\label{eq.can.G}
        \maG := (M_0 \times M_0) \cup A_{\maW}\vert_{\pa M}
\end{equation}
such that $\maG$ is a Lie groupoid integrating $A = A_{\maW}$. Thus
any Lie algebroid associated to an asymptotically commutative Lie
manifold has a canonical Lie groupoid integrating it.  Let
$\overline{A}$ be the sphere bundle obtained by radial
compactification of the fibres of $A$.

\begin{prop} \label{comm indicial alg}
Assume $(M, \maW)$ is an asymptotically commutative Lie manifold and
let $\maG$ be the canonical Lie algebroid integrating it, as in
Equation \eqref{eq.can.G}. Then $\Psi^0(\maG_{\partial M})$ is
commutative and
\begin{equation*}
        \overline{\Psi^0(\maG_{\partial M})} \cong
        \maC(\overline{A}\vert_{\partial M}).
\end{equation*}
\end{prop}

\begin{proof} 
It follows from \eqref{eq.can.G} that the algebra of
pseudodifferential operators on $\maG_{\pa M}$ coincides with $\Psi^0
(A_{\pa M})$, that is, with the algebra of continuous families of
$(P_x)$ of translation invariant pseudodifferential operators $P_x$
acting on the fibers $(A_{\pa M})_x$ of $A_{\pa M}$, $x \in \pa M$.

Now, the pseudodifferential operators of order zero on a vector space
$V$ that are translation invariant coincide with convolution operators
with functions whose Fourier transform is in
\begin{equation*}
        \tilde{S^0}(V)=\{p\in \maC^\infty(V)\colon
        p(y,\xi):=p(\xi) \in S^0(T^*V)\}
\end{equation*}
(symbols of order zero that are independent of $y$). The algebra of
convolution operators is commutative, so it follows straight away that
$\Psi^0 (A_{\pa M})$ is commutative. (In particular, the reduced and
full $C^*$-norms coincide.) Moreover, one can check that
$\overline{\tilde{S^0}(V)}\cong \maC(\overline{V})$, with
$\overline{V}$ the radial compactification of $V$. Hence,
$\overline{\Psi^0 (A_{\pa M})}\cong \maC(\overline{A}\vert_{\partial
M})$, since there is a isomorphism between elements of $\Psi^0 (A_{\pa
M})$ and continuous families in ${\tilde{S^0}(A_x)}$, which is bounded
with respect to the reduced, hence the full, norm.  This proves
$\overline{\Psi^0(\maG_\partial))} \cong \maC(\overline{A}\vert_{\partial
M})$, as claimed.
\end{proof}

Note that it follows from the proof that the isomorphism above is
really given by the total symbol, as in \eqref{pdoRn}, of the indicial
boundary operator. For order $m>0$ operators, we have
$\Psi^m(\maG_{\pa M}) = \Psi^m({A}\vert_{\pa M})\cong
{\tilde{S^m}({A}\vert_{\pa M})} \subset \maC^\infty({A}\vert_{\pa
M})$, the isomorphism being again given by the total symbol, that is,
including the lower order terms of the symbol. (This total symbol is
defined since the resulting operators on the fibers of $A \to \pa M$
are translation invariant, and hence they are convolution
operators. The total symbol is simply the Fourier transform of the
resulting convolution distributions.)

As in Section \ref{ss comp spaces}, Equation \eqref{omega} consider
$\Omega := \partial(\overline{A}) =
(S^*A) \cup \overline{A} \vert_{\partial M}$ such that $\maC(\Omega)
= \{(f,g) \in \maC(S^*{A}) \oplus \maC( \overline{A} \vert_{\partial
M}), f = g \text{ on } S^*{A}_{\partial M} \}.$

Define the \textit{boundary symbol} for operators on $(M, \maW)$ by
\begin{equation}\label{boundary.symb}
	\sigma_\pa: \overline{\Psi_{\scriptsize\maW}^0(M_0;E)} \to \maC
	(\overline{A}\vert_{\partial M})
\end{equation}
as the map of restriction to the boundary composed with the
isomorphism given by the previous proposition. For
$P\in \Psi^m_\maW(M_0;E)$, the boundary symbol is just given by the
total symbol of $\maR_\pa(Q)=Q\vert_{\pa M}\in \Psi^0(A_{\pa M})$,
with $\maR_{\pa M}: \Psi^{m}(\maG, r^*E)\to \Psi^{m}(A_{\pa M},
r^*E_{\pa M})$ the restriction map and $\pi_{M_0}(Q)=P$.

Moreover, it follows from (\ref{ex seq lie}) that $\maC(\Omega)$ is
the recipient of full symbols of pseudodifferential operators on
$M$, since $\maC(\overline{A}\vert_{\partial
M})\times_{\maC_0(S^*A_{\partial M})} \maC_0(S^*A)
=\maC(\partial\overline{A}) = \maC(\Omega)$. We have then a map
\begin{equation}
        \sigma_{full} := (\sigma_0, \sigma_\pa) :
        \overline{ \Psi_{\scriptsize\maW}^0(M;E)} \to \maC(\Omega),
\end{equation} 
which is surjective, continuous and a $*$-algebra morphism. We will
see in the next proposition that $\maK\subset \ker \sigma_{full}$, so
it follows that $\Omega=\partial(\overline{A})$ is a comparison space
for $\overline{\Psi_{\scriptsize\maW}^0(M;E)}$ (see Equation
\eqref{complete symbol}), and hence the results from
Section \ref{ss comp spaces} apply.

\begin{prop}\label{prop.Fred.comm}  
Assume $(M, \maW)$ is an asymptotically commutative Lie manifold and
let $\maG$ be the canonical Lie groupoid integrating it, as in
Equation \eqref{eq.can.G}. Then $\pi_{M_0}$ is injective on
$\overline{\Psi^0(\maG)}$, and hence the following sequence is exact.
\begin{equation}
\begin{CD}
  0 @>>> \maK(M;E) @>>> \overline{ \Psi_\maW^0(M;E)}
  @>{(\sigma_\pa,\sigma_0)}>>\maC(\Omega) @>>> 0.
\end{CD}
\end{equation}
In particular, an operator $P \in \overline{\Psi^0_{\maW}(M; E)}$ is
Fredholm if, and only if, it is fully elliptic, meaning that
$\sigma_{full}(P) = (\sigma_{0}(P), \sigma_{\pa}(P)) \in \maC(\Omega)$
is invertible.
\end{prop}

\begin{proof} 
The second part will follow from the first part
using Proposition \ref{prop.fredholm.lie}, so we concentrate on
proving the injectivity of $\pi_{M_0}$.  Let $I$ be the kernel of
$\pi_{M_0}$. We want to show that $I = \{0\}$. We have that
$\pi_{M_0}$ is injective on the subalgebra of compact operators of
$\Psi^0_\maW(M,E)$, so $I \cap \maK = 0$. It follows that
$(\sigma_0, \sigma_{\pa M})$ is injective on $I$, since it has kernel
$\maK$.

Let $P \in I \subset \Psi^0_{\maW}(M,E)$. We can recover the
principal symbol of $P$ from its action on $M_0$ \cite{AIN, Horm,
LN} so we can assume $m < 0$. By replacing $P$ with a power of $P^*P$,
we can assume that $m < -n$. The Fourier transform (as in
Equation \eqref{ker.psi}) then allows us to recover the boundary
symbol of $P$ since for $x$ approaching the boundary, the exponential
map increases its radius of injectivity (so the cutoff $\chi$ will
affect less and less the kernel of the resulting operator). This shows
that $\pi_{M_0}$ is injective on $\overline{\Psi^0(\maG)}$. 
\end{proof}

In this case, let $ [\sigma_0(P)] \in K^0(TM_0)$ denote the
$K^0$-theory class associated to $P$, as in (\ref{k0 class comp sp}),
and $[\sigma_{full}(P)]:=[(\sigma_\pa(P),\sigma_0(P))]\in
K_1(\maC(\Omega))\cong K^1(\Omega)$ denote the class in $K^1$. As
before, let $Td(T_\CC M)$ denote the Todd class of the complexified
tangent bundle of $M$, and $\pi: \overline{T^*M}\to M$. Also, we
denote by $\pi_\Omega: \Omega= \partial(\overline{A}) \to M$ the
natural projection.  From Theorem \ref{ind thm comp spaces} it finally
follows:

\begin{thm}\label{ind thm lie mflds}
Let $(M,\maW)$ be an asymptotically commutative Lie manifold manifold
with Lie algebroid $A$, $\Omega := \partial(\overline{A})$, and let
$P \in \overline{\Psi_\maW^0(M;E)}$ be an elliptic operator with
$\sigma_\pa(P)$ invertible in $\maC(\overline{A}_{\pa M})$. Then,
\begin{equation*}
	\ind(P) = (-1)^n \ch_0[\tilde{\sigma}_{full}(P)] \pi^*Td(T_\CC
	M) [TM_0] =
	{(-1)^n} \ch_1[\sigma_{full}(P)] \pi_\Omega^*Td(T_\CC M)
	[\Omega],
\end{equation*}
where $[\tilde{\sigma}_{full}(P)] \in K^0(TM_0)$ is
defined using Lemma \ref{lemma.def.ps3}.
\end{thm}

Our main example of an asymptotically commutative Lie manifold
$(M, \maW)$ is obtained as follows. Let $(M, \maV)$ be a Lie manifold
and let $x_k$ be boundary defining functions of the hyperfaces of
$M$. Choose $a_k \in \NN = \{ 1, 2, \ldots \}$. Then, as in the
Equation \eqref{eq.def.W}, we introduce
\begin{equation*}
        \maW := f \maV , \mbox{ with } f:= \Pi x_k^{a_k},
\end{equation*}
is also a structural Lie algebra of vector fields, since it is closed
for Lie brackets, and a finitely generated, projective
$\maC^\infty(M)$--module. Hence $(M,\maW)$ is a Lie manifold that is
easily seen to be asymptotically commutative.

The previous result extends the known index formulas for the
scattering calculus on manifolds with boundary, where
$\maV_{sc}:= x\maV_b$, with $x$ is a boundary defining function and
$\maV_b$ is the Lie algebra of vector fields tangent to the boundary,
and for the double-edge calculus, where $\maV_{de}=x\maV_{e}$, with
$\maV_{e}$ the edge vector fields induced by a fibration of the
boundary \cite{HunsickerGrieser, Kottke, LMor2, MelroseScattering}. Moreover,
the index formula above can be proved in the same way considering
families of pseudodifferential operators over a compact base space $B$
(the index now takes values in $K^0(B)$) using a generalization of the
Atiyah-Singer index theorem for families of asymptotically
multiplication operators. In this sense, Theorem \ref{ind thm lie
mflds} yields the result in \cite{Kottke} for families of scattering
pseudodifferential operators.

In the next section, we will apply the index formula above to compute
the index of perturbed Dirac operators on \emph{general} Lie
manifolds.



\section{Perturbed Dirac operators}\label{sec3}

Throughout this section, we let $M_0$ be a non-compact, \emph{even
dimensional} manifold $M_0$, which, as before, is assumed to be the
interior of a Lie manfold $(M, \maV)$. We fix a set $\{x_k\}$ of
defining functions of $M$ and let
\begin{equation} \label{f}
                 f: = \Pi x_k^{a_k},\ a_k \in \NN,
\end{equation}
(so $a_k > 0$).  We consider in this section a Dirac operator $\Dir$ coupled with a potential $V$, that is, an operator of the form 
\begin{equation}\label{eq.def.T}
        T = \Dir + V := \Dir \hat \otimes 1 + 1 \hat \otimes V
\end{equation} 
on compactly supported sections of some vector bundles defined on
$M_0$. By a {\em potential} we shall always mean an odd, self-adjoint
endomorphism of a $\ZZ_2$-graded vector bundle over $M_0$. An operator
$T$ of this type with will be called a {\em Callias-type
operator}. (More precisely, $T$ is the closure of $T \hat \otimes 1 +
1 \hat\otimes V$.)  We assume the potential $V$ to be of the form
\begin{equation*}
        V := f^{-1} V_0 = \Pi x_k^{-a_k} V_0,
\end{equation*} 
where $V_0$ extends to a smooth function on $M$, {\em invertible} at
the boundary. In particular, the potential $V$ is \emph{unbounded}.

We apply the results of the previous section to give a cohomological
formula for the index of $T^+ := (\Dir + V)^+.$ The main point is to
reduce the calculation of the index of $T^+$ to the case of a Dirac
operator coupled with a {\em bounded potential} on the {
asymptotically commutative} Lie manifold $(M, \maW)$ defined by
$\maW:= f\maV$, and show that the index can be obtained from
Theorem \ref{ind thm lie mflds}.  More precisely, we shall show that
\begin{equation}\label{eq.equal.ind}
        \ind(T^+) = \ind(Q) \ \text{ for }\ Q := f^{1/2} T^+
        f^{1/2} \in \Psi^1_{\maW}(M; F_0, F_1 ), \ \maW := f \maV,
\end{equation}
for suitable vector bundles $F_0$ and $F_1$. We then {use} that 
\begin{equation*}
        P := Q(1 + Q^*Q)^{-1/2} \in \overline{\Psi^0_{\maW}(M; F_0,
        F_1)}
\end{equation*} also
satifies $\ind (P^+) = \ind(Q^+)$. Finally, we show that $\ind(P^+)$, and
hence also $\ind(T^+) = \ind(P^+)$, can be computed using
Theorem \ref{ind thm lie mflds}.


\subsection{Dirac and Callias operators\label{ssec.Callias}}
Let $W$ and $E$ be $\ZZ_2$-graded vector bundles over $M$. We endow
$W \otimes E$ with the usual grading and denote by $W\hat{\otimes}E$
the resulting $\ZZ_2$-graded vector bundle, namely,
\begin{equation*}
        (W\hat{\otimes}E)^+=( W^+\otimes E^+) \oplus (W^-\otimes
        E^-) \ \text{ and } \ (W\hat{\otimes}E)^-= (W^-\otimes
        E^+) \oplus (W^+\otimes E^-).
\end{equation*}
If $V \in \End(E)$ is an endomorphism, then it acts on $C^\infty(E)$
as a (pseudo)differential operator of order $0$.

\begin{defn}\label{callias.pdo}
An operator $T: \maC^\infty_c(M_0; W\hat{\otimes}
E) \to \maC^\infty_c(M_0; W\hat {\otimes} E)$ is said to be
a \emph{Callias-type pseudodifferential operator} on the Lie manifold
$(M,\maV)$ if
\begin{equation*}
  T := D + V := D \hat{\otimes} 1 + 1\hat{\otimes} V.
\end{equation*}
where $D\in \Psi^m_\maV(M,W)$, $m>0$, is an odd, symmetric, elliptic
operator and $V\in \End(E{\vert_{M_0}})$ is {odd and self-adjoint}
{and invertible outside a compact set.} We refer to $V$ as
a \textit{potential}. {We shall also assume our potential $V$ to be
invertible outside a compact subset of $M_0$.} The closure of an
operator of the form $T = D + V$ will also be called a {\em
Callias-type operator}.
\end{defn}

When $D$ is the (generalized) Dirac operator, these operators are also
called Dirac-Schr\"{o}dinger operators and were first considered by
Callias (in the odd dimensional Euclidean space \cite{Callias}). See
also \cite{bunke, bunkeSchick, FoxHaskell1, FH2} and references
therein for more results on index theory of Dirac-Schr\"{o}dinger and
Callias type operators on even-dimensional manifolds.

\begin{rmk}
On odd dimensional manifolds, the Callias-type operators are of the
form $\Dir + iV$, where $V$ is self-adjoint and invertible at
infinity.  See \cite{Anghel1, Anghel2, BottSee,BM, Callias, FH3,
Kottke, rade} for more on the index of Calias type operators in the
odd-case.
\end{rmk}

Recall that a symmetric (hence closable) operator $T$
is \emph{essentially self-adjoint} if its closure is self-adjoint,
that is, if $<{T}x,y>=<x,{T}y>$, for all
$x,y \in \maD(T)=\maD(T^*)$. (We shall always denote the minimal
closure of an operator by the same letter.)

In the following lemma, we assume that the potential $V_0$ extends to $M$,
in particular, it is bounded. (We will prove such a result for an 
unbounded potential in Section \ref{s.pot.unbdd}.)

\begin{lem} \label{lemma.sa} 
Let $D\in \Psi^m_\maV(M,W)$, $m>0$, be an odd, symmetric, elliptic
operator. Assume that $V_0$ extends to a smooth function on $M$, as
before, then the Callias-type operator $T =D +
V_0 \in \Psi^m_{\maV}(M_0; W \otimes E)$ is elliptic and essentially
self-adjoint on $\CIc(M_0; W\otimes E)$.
\end{lem}

\begin{proof}
Ellipticity follows from $\sigma_{m}(T) = \sigma_{m}(D)$.  The fact
that $T$ is essentially self-adjoint follows, for instance,
from \cite{LN} (Theorem 7.1) which yields that, with $m>0$, a
(possibly unbounded) symmetric, elliptic operator in $\Psi^m_{\maV}(M;
W\otimes E)$ is essentially self-adjoint, identifying
$\Psi^m_{\maV}(M; W\otimes E) = \pi_{M_0}(\Psi^m(\maG; r^*(W\otimes
E)))$, as in Theorem \ref{psiV as psiG}.
\end{proof}

We shall work with unbounded Fredholm operators.  It will then be
useful to recall the way they are introduced. Let $T$ be a possibly
unbounded operator with domain $\maD(T)$ and codomain $H$.  {We shall
always replace $T$ by its closure, so assume $T$ is closed} and endow
$\maD(T)$ with the graph norm. Then $T$ is Fredholm if, by definition,
the induced bounded operator $T : \maD(T) \to H$ is Fredholm (in the
usual sense of having finite dimensional kernel and cokernel). In
particular, a pseudodifferential operator $T_1$ acting between
sections of $E_0$ with range sections of $E_1$ is Fredholm if, and
only if, $T_2 := T_1(1+T_1^*T_1)^{-1/2}$ is a Fredholm operator and,
in this case, $\ind(T_1)=\ind(T_2)$.

We are interested in
computing the index of
\begin{equation}\label{eq.def.T+}
        T^+ =(D + V)^+ : \maC^\infty_c(M_0; (W \hat{\otimes}
        E)^+) \to \maC^\infty_c(M_0; (W\hat{\otimes} E)^-),
\end{equation}
which we shall prove to be Fredholmness between suitable Sobolev
spaces.

Note that, with respect to the grading, we can write
\begin{equation}\label{D.graded}
        T^+ = \left(\begin{array}{cc} D^+\otimes 1 & -1 \otimes V^- \\
        1 \otimes V^+ & D^-\otimes 1 \end{array}\right).
\end{equation}
Most of our results work for general odd, elliptic, positive
pseudodifferential operators $D \in \Psi^m_{\maV}(M; E)$. However, for
simplicity and because this is the most useful case in applications,
we shall mainly be interested in the case when $D$ is a generalized Dirac
operator. Recall that, in any case, the Dirac operators generate all
classes in $K$-homology, so we can always assume $D$ to be a Dirac
operator.


\subsection{Dirac operators on Lie manifolds\label{ss unbdd pot lie}} 
We introduce here generalized Dirac operators on Lie manifolds
following \cite{ALNgeo}. Let $(M, \maV)$ be a {\em even} dimensional
Lie manifold endowed with a {\em compatible} metric $g$ on $M_0$ and
let $W$ be a Clifford module over $M$ endowed with an $A^*$-valued
connection $\nabla^W$ and a Clifford multiplication bundle map $c:
A\otimes W \to W$. Recall that a {\em compatible} metric on $M_0$ is a
metric on $TM_0$ that extends to $A \to M$. The restrictions of $W$,
$c$, and $\nabla^W$ to $M_0$ reduce to the classical notions of a
Clifford bundle together with an admissible connection \cite{Gilkey,
LawsonBook, MelroseScattering}.

\begin{defn} 
The \textit{generalized Dirac operator} $\Dir : \maC^\infty(M;W) \to
\maC^\infty(M;W)$ associated to $W$ is then defined as the composition
\begin{equation}\label{dir}
\begin{CD}
        \maC^\infty(M;W)@>\nabla^W>> \maC^\infty(M;W\otimes A^*)
        @>id\otimes \phi>> \maC^\infty(M;W\otimes
        A)@>c>> \maC^\infty(M;W),
\end{CD}
\end{equation}
where $\phi: A^*\to A$ is the isomorphism given by the metric.  
\end{defn}

Since both the Clifford multiplication $c$ and the $A^*$-valued
connection are $\maV$-differential operators, of order $0$ and $1$,
respectively, we have that $\Dir\in \Diff_{\maV}^1(M;W)$. The
principal symbol $\sigma_1(\Dir)\xi = ic(\xi)\in \End(W)$ is invertible
for any $\xi\neq 0$, and hence $\Dir$ is elliptic.  It follows from
classical results that $\Dir$ with domain $\maC^\infty_c(M_0;
W) \subset L^2(M_0; W)$ is essentially self-adjoint (\ie its closure
is self-adjoint), since $M_0$ is complete.

We can also define Dirac operators on groupoids: if $\maG$ is a
$d$-connected groupoid integrating $A = A(\maV)$  then we can consider the
Clifford module $r^*W$ and endow $\maG$ with an admissible connection
$\nabla^\maG\in \Diff(\maG; W, W\otimes A^*)$ such that
$\pi_{M_0}(\Dir^\maG)=\Dir$, where $\Dir^\maG$ is the associated Dirac
operator on $\maG$ (see \cite{LN} for details).

Assume now that $M$ is even-dimensional and $W$ is $\ZZ_2$-graded,
with the grading given by the chirality operator. Let also $E$ be an
Hermitian $\ZZ_2$-graded vector bundle over $M$ and $V\in\End(E)$ a
potential (so odd, self-adjoint). We are interested in computing the
index of
\begin{equation}
        T^+ =(\Dir + V)^+ : \maC^\infty_c(M_0; (W \otimes
        E)^+) \to \maC^\infty_c(M_0; (W\hat\otimes E)^-).
\end{equation}


\subsection{The case of bounded potentials}\label{s.pot.bdd}
Let $(M, \maW)$ be an asymptotically commutative Lie manifold. 
Recall that in this section we assume that $n$, the 
dimension of $M$, is even.

Let 
\begin{equation}\label{eq.def.Q.1}
        Q := D + V_0 \text{ with } D \in \Psi^m_{\maW}(M_0; E\otimes
        W),
\end{equation}
where $D$ is an elliptic, symmetric, odd pseudodifferential operator,
as in Definition \ref{callias.pdo}. Let $V_0$ be a bounded potential
on $M_0$ that extends to a smooth function on $M$ that is invertible
on $\pa M$ (so, in particular, $V_0$ is odd and symmetric). 
 It follows from
Lemma \ref{lemma.sa} that $Q$ is elliptic and essentially self-adjoint
on $\CIc(M_0)$.

We define the total symbol $K$-theory classes $\sigma_{full}(Q) \in
K^1(\Omega)$ and $\tilde{\sigma}_{full}(Q) \in K^0(TM_0)$ of $Q$ in a
similar way to the case of order zero symbols. First, recall that the
boundary symbol $\sigma_\pa:
\Psi_{\scriptsize\maW}^m(M_0;E) \to \maC({A}\vert_{\partial
M})$ is given by 
\begin{equation*}
        \sigma_\pa(P) = \sigma_m^{tot}(\maR_\pa(S))
        = \sigma_m^{tot}(S\vert_{\pa M}),
\end{equation*} with $\pi_{M_0}(S)=P$,
$\maR_\pa: \Psi^m(\maG)\to \Psi^m(\maG_{\pa M})$ is restriction to the
boundary and $\sigma_m^{tot}(S_x)$, $x\in \pa M$, is the total symbol
of the operator $S_x$ on $A_x$ (including lower order terms).

\begin{lem}\label{lemma.fe}
Let $Q$ be as in Equation \eqref{eq.def.Q.1} and $P :=
Q(1+Q^2)^{-1/2}$. Then $P \in \overline{\Psi^0_{\maW}(M_0; W \otimes
E)}$ is fully elliptic, in the sense that its principal symbol
$\sigma_0(P)$ and the boundary symbol $\sigma_{\pa}(P)$, defined by
continuity, are invertible.
\end{lem}

\begin{proof} 
It follows from Lemma
\ref{lemma.order.red} that $P \in \overline{\Psi^0_{\maW}(M_0;
W \otimes E)}$. We have that $P$ is elliptic, since $Q$ is.  To
understand the boundary operators, since $\maW$ is commutative at the
boundary, hence $\maG_{\pa M} = A_{\pa M}$ is amenable, we only need
invertibility on fibres $\maG_x=A_x$, $x\in \pa M$ (see the remark
after Theorem \ref{prop.fredholm.lie}).  {Let $S\in \Psi^m(\maG;
r^*W)$ be such that $\pi_{M_0}(S)=D$}. Therefore, we need to look at
the symbol of the operators {$S_x$} coupled with the constant
potential $V_0(x)$ acting on the fiber $A_x$, for each $x \in \pa
M$. The invertibility of the boundary indicial operator $\sigma_\pa(D
+ V_0)_x(\xi) = \sigma^{tot}(S_x) \hat{\otimes}1 + 1\hat{\otimes}
V_0(x)$, $\xi \in A_x$ then follows from the fact that $V_0(x)$ is
invertible for each $x \in \pa M$ (noting that $\alpha\hat{\otimes}1 +
1\hat{\otimes} \beta\in \End(W_x\hat\otimes E_x)$ is invertible if
$\alpha$ or $\beta$ are invertible.)
\end{proof}

Note that when $D=\Dir$ is a Dirac operator on $(M, \maW)$ then, using
the notation as above, $S_x=\Dir_x$ is a Dirac operator on $A_x$ and
$\sigma_\pa(\Dir + V)_x(\xi)= ic(\xi) \hat{\otimes} 1 +
1 \hat{\otimes} V_0(x)$. (This is due to the fact that the restriction
of a Dirac operator to the boundary is again a Dirac
operator \cite{LN}.)

The following lemma provides the definitions of the total symbol
$K$-theory classes $\sigma_{full}(Q) \in K^1(\Omega)$ and
$\tilde{\sigma}_{full}(Q) \in K^0(TM_0)$.  Let us introduce the
$K$-theory class $[V_0]$ defined by the endomorphism $V_0$ as usual
\cite{AtiyahBook, KaroubiBook}
\begin{equation*}
        [V_0]:= [E^+, E^-, V_0] \in K^0(M_0) = K^0(M;\pa M)\subset
        K^0(M).
\end{equation*}
Recall that $\Omega := \partial(\overline{A}) =
(S^*A) \cup \overline{A} \vert_{\partial M}$ (as in Subsection \ref{ss
ind form lie}).

\begin{lem}\label{lemma.symb.Q}
Let $Q$ and $P$ be as in Lemma \ref{lemma.fe} and define
$[\sigma_{full}(Q)] := [\sigma_{full}(P)] \in K^1(\Omega)$ and
$[\tilde{\sigma}_{full}(Q)] := [\tilde{\sigma}_{full}(P)] \in
K^0(TM_0)$.  Then
\begin{equation*}
        \pa [\sigma_{full}(Q)] = [\tilde{\sigma}_{full}(Q)]
\end{equation*} 
and $[\tilde{\sigma}_{full}(Q)]$ can be represented by the
endomorphism $\sigma_m(D) \hat\otimes 1 + 1 \hat\otimes V_0$.  In
particular,
\begin{equation*}
        [\tilde{\sigma}_{full}(Q)] = [\sigma_m(D)] \otimes \pi^*[V_0],
\end{equation*} 
where $[\sigma_m(D)] \in K^0(TM)$ and $[V_0] \in K^0(M, \pa M)$ are
the classes defined by the corresponding morphisms and $\pi : TM \to
M$ is the natural projection.
\end{lem}

\begin{proof} 
The relation $\pa [\sigma_{full}(Q)] = [\tilde{\sigma}_{full}(Q)]$
follows from definitions and from Lemma \ref{lemma.def.ps3}.  Let us
choose a smooth function $\sigma_m  \in S^m(A^*)$ such that $\sigma_m$
represents $\sigma_m(D)$ and on $A^*\vert_{\pa M}$ it is equal to the
total symbol of $D$. Let then $p = \sigma_m \hat\otimes 1 +
1 \hat\otimes V_0 \in \CI(A^*) = \CI(A)$, where we have used a fixed
metric on $A$ to identify $A$ with $A^*$, as before. From
Equation \eqref{D.graded} we have that
\begin{equation*}
        \sigma_0(P) = \sigma_m(Q)/\sqrt{1 + \sigma_m(Q)^2}
        = \sigma_m(D)/\sqrt{1 + \sigma_m(D)^2} = p/\sqrt{1 + p^2} \in
        S^0(A^*)/S^{-1}(A^*).
\end{equation*}
On the other hand, at the boundary, we have
\begin{equation*}
        \sigma_\pa(P) = \sigma_\pa(Q)/\sqrt{1 + \sigma_\pa(Q)^2} =
        p/\sqrt{1 + p^2}.
\end{equation*}
Therefore, $\sigma_{full}(P) = p/\sqrt{1 + p^2}$ on $\Omega$. Hence
the $K$-theory class $[\tilde{\sigma}_{full}(P)]$ is obtained from the
endomorphism $p/\sqrt{1 + p^2}$ defined on $TM_0$, which obviously
extends $p/\sqrt{1 + p^2}$ from $\Omega$ to the whole of
$\overline{A} \supset TM_0$. We obtain that the endomorphism $p$, and hence also $\sigma_m(D) \hat\otimes 1 + 1 \hat\otimes V_0$,  represents
$[\tilde{\sigma}_{full}(Q)]$.

{F}rom the definition of tensor product in $K$-theory we have that
\begin{equation*}
        [\sigma_{m}(D)] \otimes \pi^*[V_0] = [\pi^*(W\hat\otimes
        E)^+, \pi^*(W\hat\otimes E)^-, \sigma_{m}(D)\hat\otimes 1 \oplus
        1 \hat \otimes V_0]
\end{equation*}
where
\begin{equation*}
        \sigma_{m}(D)\hat\otimes 1 \oplus 1\hat\otimes V_0
        = \left(\begin{array}{cc} \sigma_m(D^+)\otimes 1 & -1 \otimes
        V_0^- \\
        1 \otimes V_0^+ & \sigma_m(D^-)\otimes
        1 \end{array}\right).
\end{equation*}
It follows that the $K$-theory class $[\tilde{\sigma}_{full}(Q)]$,
where $Q=D \hat{\otimes}1 + 1 \hat{\otimes} V_0$, is represented by
the same morphism as $[\sigma_m(D)] \otimes \pi^*[V_0]$. So these two
classes are equal.
\end{proof}

We shall need the Sobolev spaces $H_{\maW}^m(M_0)$ defined by $\maW$
(more precisely by the metric determined by $\maW$ \cite{ALNpdo,
AIN}).
\begin{equation} 
        H_{\maW}^m(M_0) := \{u\in L^2(M),\ Du\in L^2(M_0) \text{ for
        all } D\in \Diff^m_\maW(M_0)\}.
\end{equation}
The space $H_{\maW}^m(M_0)$ is the domain of any elliptic
pseudodifferential operator in $\Psi_{\maW}^m(M_0)$, $m >0$, acting
on $L^2(M_0)$. For $m < 0$ we use duality.

We now show that $\ind(Q^+)$ can be computed using Theorem \ref{ind
thm lie mflds}.

\begin{thm}\label{thm.b.potential} 
Let $Q = D + V_0$ be a Callias-type pseudodifferential operator with a bounded
potential $V_0$ as in Lemmas \ref{lemma.fe} and \ref{lemma.symb.Q}.
In particular, we assume that $V_0$ is a smooth potential on $M$
that is invertible on $\pa M$. Then $Q^+$ is Fredholm and
\begin{multline*}
	\ind(Q^+) = \ch_0[\tilde{\sigma}_{full} (Q^+)] \pi^*Td(T_\CC
	M) [TM_0] = \ch_1[\sigma_{full}(Q^+)] \pi_\Omega^*Td(T_\CC M)
	[\Omega]\\
	 = \ch_0[\sigma_{m}(D^+)]\ch_0 {\pi^*} [V_0] \pi^*Td(T_\CC
	 M) [TM_0].
\end{multline*} 
\end{thm}

\begin{proof} 
Let $P := Q(1+Q^*Q)^{-1/2}$, as before. Then
$P \in \overline{\Psi^0_{\maW}(M_0; W \otimes E)}$ is fully elliptic,
by the previous lemma (Lemma \ref{lemma.fe}). Hence $P$ is Fredholm by
Proposition \ref{prop.fredholm.lie}.
\begin{multline*}\label{eq.mult.many}
	\ind(Q^+) = \ind(P^+) \\
        = \ch_0[\tilde{\sigma}_{full}(P^+)]_0 Td(T_\CC M_0) [TM_0]
	= \ch_1[\sigma_{full}(P^+)] \pi^*Td(T_\CC) [\Omega]\\
        = \ch_0[\tilde{\sigma}_{full}(Q^+)] Td(T_\CC M_0) [TM_0]
	= \ch_1[\sigma_{full}(Q^+)] Td(T_\CC) [\Omega],\\
	 = \ch_0[\sigma_{m}(D^+)]\ch_0 {\pi^*} [V_0] \pi^*Td(T_\CC
	 M) [TM_0].
\end{multline*} 
by Theorem \ref{ind thm lie mflds} applied to $P^+$ and Lemma
\ref{lemma.symb.Q}.
\end{proof}

We are mainly interested in the case when
\begin{equation}\label{bdd pot}
       Q = \Dir + V_0 := \Dir \hat{\otimes} 1 + 1 \hat{\otimes} V_0,
\end{equation}
where $\Dir$ is a Dirac operator acting on the sections of some
Clifford bundle $W$. As before, we assume $V_0$ is potential (\ie an
odd, self-adjoint, endomorphism of a $\ZZ_2$-graded bundle $E$) that
is invertible outside a compact subset of $M_0$ such that $V_0$
extends smoothly to $M$, to be invertible at $\pa M$. In particular
$Q\in \Psi^1_{\maW}(M_0; W \otimes E)$.

To get an even more explicit formula for the index of coupled Dirac
operators $\Dir + V_0$, let us now that $M$ has a spin$^c$-structure,
with canonical spin$^c$-bunde $S$ and associated Dirac operator
$\Dir_S$. In particular, $M$ is oriented, and we let $[M] \in
{H_n(M, \pa M) }$ denote its fundamental class.  Let $W =
S \hat \otimes F$, with $F$ a complex vector bundle over $M$. Then
$\Dir_F := \Dir_{S} \otimes F$ is \emph{the Dirac operator twisted
with $F$.}

\begin{cor}\label{cor.tw.dirac} Let $\Dir_F$ be the Dirac operator
twisted with $F$ and $Q = \Dir_F + V_0$ be the perturbed twisted Dirac
operator associated to $V_0$, where $V_0$ is a bounded potential
invertible at $\pa M$ on an asymptotically commutative spin$^c$ Lie
manifold $(M, \maW)$. Then $Q^+$ is Fredholm and
\begin{equation*}	
      \ind(Q^+) = {\hat A(M) \ch_0 ([F \otimes V_0] ) [M]}.
\end{equation*}
\end{cor}

\begin{proof}
It is known classically that
\begin{equation*}
      p_!  \ch_0(\sigma(\Dir_F^+)) Td(T_\CC M) =\hat A(M)  \ch_0[F],
\end{equation*}
where $p_! $ is integration over the fibre and $\hat A(M)\in H^*(M)$
is the $\hat{A}$-genus of $M$ (see \cite{LawsonBook}). The result
then follows right away from Theorem \ref{thm.b.potential}.
\end{proof}



\subsection{The case of unbounded potentials}\label{s.pot.unbdd}
In this subsection, we are back to a general (even-dimensional) Lie
manifold $(M, \maV)$. Let 
\begin{equation}
\label{unbdd.pot}  
    T := \Dir + V := \Dir \hat{\otimes} 1 + 1 \hat{\otimes} V,
\end{equation} 
where $\Dir\in \Diff^1_\maV(M_0; W)$ is a Dirac operator associated to
$\maV$.  We will consider here
\textit{unbounded potentials}, in that we assume moreover that,
on $M_0$,
\begin{equation}\label{pot.unbdd}
  V = f^{-1}V_0, \quad f:=\Pi x_k^{a_k},
\end{equation}
where $V_0$ is bounded and it extends to a smooth function on $M$ that
is invertible on $\pa M$ (at infinity) and $x_k$ are boundary defining
functions of the hyperfaces of $M$ with $a_k\in \NN = \{ 1,
2, \ldots \}$. This section contains the hard analysis needed
for our main result.

Our first goal is to show that $T$ is essentially self-adjoint with
domain a suitable weighted Sobolev space. We want to prove a formula
for the index of $T^+ = (\Dir + V)^+$.  Our strategy is to reduce this
problem to a question on operators with bounded potential by writing
\begin{equation}\label{eq.def.Q}
  T = f^{-1/2}Qf^{-1/2}, \quad \mbox { with }\ Q:=f^{1/2}\Dir f^{1/2} +
    V_0.
\end{equation}
In fact, let $\maW:= f\maV$ and let $g$ be the given metric compatible
with $\maV$. Then $g_0 := f^{-2}g$ is a metric compatible with $\maW$
and hence
\begin{equation*}
        \Dir_\maW := f^{1/2}\Dir f^{1/2} 
\end{equation*} 
is the Dirac operator associated to the Lie manifold structure defined
by $\maW$ and metric $g_0$ \cite{Baer1, Baer2, Hitchin, LawsonBook,
NistorDiracEquiv}. Actually, to identify $\Dir_{\maW}$ with
$f^{1/2}\Dir f^{1/2}$, we need to rescale the volume forms also, a
fact that we ignore throughout, in order to simplify the notation.

We then have that $Q \in{\Diff^1_{\scriptsize\maW}}(M; W\otimes E)$ is
a Callias-type Dirac operator on $(M,\maW)$ with a bounded, invertible
potential. In particular, it is elliptic and essentially self-adjoint
on $\maC_c^\infty(M_0).$ As before, we still denote its self-adjoint
closure by $Q$.

We now define \emph{weighted Sobolev spaces} defined by $\maW$
\begin{equation}
\label{def.weighted.sobolev}
\maK_a^m(M_0) : = f^{a}H_{\maW}^m(M_0),
\end{equation}
where $a\in \RR$ {and $f = \prod x_k^{a_k}$, as before}. If $E \to M$
is a smooth vector bundle, then the spaces $\maK_{a}^m(M_0; E)$ are
defined similarly. We remark that {all} the weighted Sobolev spaces
used below are with respect to $\maW$. (One can check that
$\maK_{a}^m(M_0; E; \maV)=\maK_{a-n/2}^m(M_0; E;\maW)$.)  We have the
following elliptic regularity result from \cite{ALNpdo}.

\begin{thm}\label{thm.reg}
Assume that $Q_0 \in \Psi_{\maW}^k(M_0;E)$ is elliptic and $h
\in \maK_a^s(M_0; E)$ is such that $Q_0h \in \maK_a^{m-k}(M_0; E)$. 
Then $h \in \maK_a^m(M_0; E)$.
\end{thm}

Applying this result to $Q = f^{1/2} T f^{1/2}=\Dir_\maW + V_0$ we
obtain the following.

\begin{lem} \label{lemma.reg}
Let $h \in \maK_a^s(M_0; E)$ be such that
$Th \in \maK_{a-1}^{m-1}(M_0; E)$.  Then $h
\in \maK_a^m(M_0; E)$.
\end{lem}

We shall also need the following lemma. Before, we remark that it
follows from the definitions that multiplication by $f^s$ defines an
isomorphism $f^s: \maK_a^m(M_0; E_0) \to \maK_{a+s}^m(M_0;E_0)$, for
any $a, s$. {In particular, if $P\in \Psi_\maW^k(M_0)$, with $P:
H^m_\maW(M_0;E_0)\to H^{m-k}_\maW(M_0;E_0)$, then
$f^sPf^{-s}: \maK^m_s(M_0; E_0)\to \maK_s^{m-k}(M_0; E_0)$ is Fredholm
if, and only if, $P$ is. Moreover, it is known that
$f^s\Psi_\maW^k(M_0;E_0)f^{-s} = \Psi_\maW^k(M_0;E_0) $ (Proposition
4.3 \cite{ALNpdo}), so any such $P$ is also defined as an operator,
still denoted by $P$,
$$ P: \maK^m_s(M_0;E_0) \to \maK^{m-k}_s(M_0;E_0),$$ 
for any $s$.} 

\begin{lem} 
\label{lemma.indep}
Let $Q_0 \in \Psi_{\maW}^k(M_0; E_0)$, $k \in \ZZ_+$, be a fully
elliptic.  Then
\begin{equation*}
        Q_{a,b,c} := f^b Q_{{0}} f^{c} : \maK_a^m(M_0;
        E_0) \to \maK_{a+b+c}^{m-k}(M_0; E_0)
\end{equation*}
is Fredholm and its index is independent of $m$, $a$, $b$, and $c$, in
the sense that
\begin{equation*}
        \ind(Q_{a, b, c}) = \ind(Q_{0,0,0}).
\end{equation*}
\end{lem}

\begin{proof}
Let us notice first that $Q_{0,0,0}$ is Fredholm due to Proposition
\ref{prop.Fred.comm} (since we assumed $Q_0$ to be fully elliptic). 
It follows that $f^s Q_{0,0,0} f^{-s}: \maK^m_s(M_0; E_0)
\to \maK_s^{m-k}(M_0; E_0)$ is also Fredholm and $\ind(f^s Q_{0,0,0} 
f^{-s}) = \ind(Q_{0,0,0})$. Note also that $Q_{a,b,c}$ is indeed
well-defined, by the remarks above.  Write $Q_a=Q_{a,0,0}$.

Next, we notice that $f^s P f^{-s} - P = f^s(P f^{-s} - f^{-s}P) \in
f \Psi_{\maW}^{k-1}(M_0; E_0)$ for any $P\in
\Psi_{\maW}^k(M_0; E_0)$ {by the specific form of the Lie
algebra of vector fields $\maW = f\maV$. } Moreover $f^s
Q_{{a-s}} f^{-s} - Q_{{a}} :
\maK_a^m(M_0; E_0) \to \maK_a^{m-k}(M_0; E_0)$ is compact, 
since $f \maK_a^m(M_0; E_0) \to \maK_a^{m-k}(M_0; E_0)$ is compact
by \cite{AIN}. With $s=a$, we conclude that $Q_a$ is Fredholm and
also that $Q_a$ and $f^s Q_{a-s} f^{-s}$ have the same index for any
$s$.

If follows that the index of $Q_a : \maK_a^m(M_0; E_0) \to
\maK_a^{m-k}(M_0; E_0)$ is independent of $a$. Using this with $a$ replaced
by $a+c$ and using the fact that $f^s : \maK_a^m(M_0; E_0) \to
\maK_{a+s}^{m}(M_0; E_0)$ is an isomorphism, we obtain the desired result.
\end{proof}

We shall use this lemma to prove the following crucial result.

\begin{prop}\label{prop.domain}
The operators 
\begin{equation*}
        T \pm i I = \Dir + V \pm i I : \maK_1^1(M_0; W \otimes
        E) \to \maK_0^0(M_0; W \otimes E)= L^2(M_0; W \otimes E)
\end{equation*}
are invertible, and hence $T$ is essentially self-adjoint with domain
$\maK_1^1(M_0; W \otimes E)$, {where all the $L^2$ and Sobolev
spaces are associated to $\maW$. }
\end{prop}

\begin{proof}
Let us denote by 
\begin{equation*}
        Q_a := f^{1/2}Tf^{1/2} = \Dir_{\maW} + V_0
        : \maK_{a}^{1/2}(M_0; W \otimes E)\to \maK_{a}^{-1/2}(M_0;
        W \otimes E).
\end{equation*}
Then $Q_{{0}}$ is fully elliptic (by Lemma \ref{lemma.fe} and the
fact that $\Dir_{\maW}$ is the Dirac operator associated to $\maW =
f\maV$), and hence it is Fredholm. It follows from
Lemma \ref{lemma.indep} that $Q_a$ is Fredholm for any $a$ and that
its index is independent of $a$. Since $Q_0^* = Q_0$, we have
$\ind(Q_a) = 0$ for all $a$. Hence
\begin{equation*}
        \ind(Q_a + \lambda f) = \ind(Q_a)=0,
\end{equation*} 
since multiplication by $f$ is a compact operator $\maK_{a}^{1/2}(M_0;
W \otimes E) \to \maK_{a}^{-1/2}(M_0; W \otimes E)$
by \cite{AIN}. Then
\begin{equation}\label{eq.def.domain}
        T \pm i I = f^{-1/2}(Q_a \pm i f) f^{-1/2}
        : \maK_{1/2}^{1/2}(M_0; W \otimes
        E) \to \maK_{1/2}^{-1/2}(M_0; W \otimes E)
\end{equation}
is also Fredholm of index zero. 

We recall that $\maK^{m}_{a}(M_0; W \otimes E)$ is
the dual of $\maK^{-m}_{-a}(M_0; W \otimes E)$, with the duality
pairing being obtained from the $L^2$-inner product by continuous
extension. Then the ``$L^2$-estimate''
\begin{equation*}
        ( (T \pm i)u, u) = (Tu, u) \pm i(u, u) 
\end{equation*}
and $(Tu, u) \in \RR$ (since $T$ is symmetric between the indicated
spaces in Equation \eqref{eq.def.domain}) show that $T \pm i I$ are
injective for $a = 0$. Since they have index zero, they induce
isomorphisms
\begin{equation}\label{eq.isom.T}
        T \pm i I : \maK_{a + 1/2}^{m+1/2}(M_0; W \otimes E) \to \maK_{a
        -1/2}^{m - 1/2}(M_0; W \otimes E)
\end{equation}
for $a = 0$ and $m=0$. 

Now for an arbitrary $a$, the induced operator will still have index
zero (by Lemma \ref{lemma.indep}). Since for $a \ge 0$ it will still
be injective, it follows that it will be an isomorphism for all $a \ge
0$. Since for $a<0$ the resulting map is dual to the one for $-a$, we
obtain that $T \pm i I$ of Equation \eqref{eq.isom.T} are isomorphisms
for all $a$ and $m=0$. We can extend this isomorphism to any $m \ge$
by elliptic regularity (Lemma \ref{lemma.reg}) and this completes the
proof by taking $a = m = 1/2$.
\end{proof}

We shall extend $T$ to a self-adjoint operator denoted by the same
letter. We are ready now to compute the index of
\begin{equation*}
        T^+ = (\Dir + V)^+ : \maK_1^1(M_0; (W\hat\otimes
        E)^+) \to \maK^0_0(M_0; (W\hat\otimes E)^-),
\end{equation*}
where $\Dir$ is the Dirac operator on the (arbitrary) even dimensional
Lie manifold $(M_0, \maV)$ and $V=f^{-1}V_0$ is an unbounded potential
as in \eqref{pot.unbdd}. Let $\pi: \overline{TM} \to M$ and
$\pi_\Omega: \Omega = \pa{\overline{A_\maV}} \to M$ be the natural
projections and $Td(T_\CC M)$ be the Todd class of the complexified
tangent bundle of $M$.

\begin{thm}\label{thm.u.potential} 
The operator $T^+ = (\Dir + V)^+$ is Fredholm and 
its index is given by
\begin{eqnarray*}
	\ind(T^+) & =
	&\ch_0[\tilde{\sigma}_{full}(T^+)] \pi^*Td(T_\CC M) [TM_0]
	= \ch_1[\sigma_{full}(T^+)] \pi_\Omega^*Td(T_\CC M) [\Omega]\\
	& = & \ch_0[\sigma_1(\Dir^+)]\ch_0 {\pi^*} [V_0] \pi^*Td(T_\CC
	M) [TM_0].
\end{eqnarray*} 
\end{thm}

\begin{proof}
Let  $Q_1 = T^+f$, where $f = \prod x_k^{a_k}$ as above and is
regarded as a multiplication operator. Then $Q_1 =
f^{-1/2}Q^+f^{1/2}$, where
\begin{equation*}
      Q^+ := f^{1/2} T^{+} f^{1/2} = (\Dir_{\maW} + V_0)^+
      : \maK_{0}^1(M_0; (W\hat\otimes E)^+) \to \maK_0^0(M_0; (W\hat\otimes
      E)^-)
\end{equation*} 
is fully elliptic (by Theorem \ref{thm.b.potential} and the fact that
$\Dir_{\maW}$ is the Dirac operator associated to $\maW := f\maV$).
Then the operators $Q_1$ and $Q^+$ have the same index, by
Lemma \ref{lemma.indep}. By ellipticity,
\begin{equation*}
        (1 + Q_1^*Q_1)^{1/2} : \maK_0^1(M_0; (W\hat\otimes E)^+) \to
        L^2(M_0; (W\hat\otimes E)^+)
\end{equation*}
is an isomorphism since the domain of any elliptic operator $P \in 
\Psi^m_{\maW}(M; E)$ is $H^m_\maW(M_0, E)$. Therefore $f(1 + Q_1^*Q_1)^{-1/2} 
: L^2(M_0; (W\hat\otimes E)^+) \to \maK_1^1(M_0; (W\hat\otimes E)^+)$
is an isomorphism as well.  Proposition \ref{prop.domain} then yields
that $T^+ : \maK_{1}^1(M_0) \to L^2(M_0)$ has the same index as
\begin{equation*}
        T^+ f(1 + Q_1^*Q_1)^{-1/2} = Q_1(1 + Q_1^*Q_1)^{-1/2}
        : L^2(M_0; (W\hat\otimes E)^+) \to L^2(M_0; (W\hat\otimes E)^-),
\end{equation*}
and, in particular, they are both Fredholm.

We have thus obtained that the operator $T^+ = (\Dir + V)^+$ is
Fredholm and has the same index as $Q^+ := f^{1/2} T^{+}
f^{1/2}$. Moreover, the principal symbols of $T^+$ and $Q^+$ define
the same $K$-theory classes, by homotopy invariance, as do the symbols
of $\Dir$ and $\Dir_\maW$, and hence
\begin{multline*}
	\ind(T^+) = \ind(Q^+) = \ch_0[\sigma(Q^+)] \pi^*Td(T_\CC M)
        [TM_0] = \ch_1[\sigma(Q^+)] \pi_\Omega^*Td(T_\CC M) [\Omega] \\
        = \ch_0[\sigma(T^+)] \pi^*Td(T_\CC M) [TM_0]
	= \ch_1[\sigma(T^+)] \pi_\Omega^*Td(T_\CC M) [\Omega]\\
	=\ch_0[\sigma(\Dir^+)]\ch_0 {\pi^*} [V_0] \pi^*Td(T_\CC M)
	[TM_0],
\end{multline*} 
by Theorem \ref{thm.b.potential} applied to $Q^+$ and homotopy
invariance.
\end{proof}

We also obtain the following more explicit calculation similar
to Corollary \ref{cor.tw.dirac}.

\begin{cor}\label{cor.tw.dirac2} 
Let $\Dir_F$ be the Dirac operator
twisted with $F$ and $T = \Dir_F + V$ be the perturbed twisted Dirac
operator associated to $V = f^{-1}V_0$, where $V_0$ is a bounded
potential on $M$ invertible at $\pa M$ for a spin$^c$ Lie manifold
$(M, \maV)$. Then $T^+$ is Fredholm and, using the notation
of Corollary \ref{cor.tw.dirac2}
\begin{equation*}	
      \ind(T^+) = \hat A(M) \ch_0 ([F \otimes V_0] ) [M] = \hat
      A(M) \ch_0 ([F \otimes V] ) [M].
\end{equation*}
\end{cor}


\bibliographystyle{plain}
\bibliography{dirpotref}

\begin{thebibliography}{10}

\bibitem{ALNpdo}
B.~Amman, R.~Lauter, and V.~Nistor.
\newblock Pseudodifferential operators on manifolds with a {L}ie structure at
  infinity.
\newblock {\em Ann. of Math. (2)}, 165(3):717--747, 2007.

\bibitem{AIN}
B.~Ammann, A.~D. Ionescu, and V.~Nistor.
\newblock Sobolev spaces and regularity for polyhedral domains.
\newblock {\em Documenta Mathematica}, 11(2):161--206, 2006.

\bibitem{ALNgeo}
B.~Ammann, R.~Lauter, and V.~Nistor.
\newblock On the geometry of {R}iemannian manifolds with a {L}ie structure at
  infinity.
\newblock {\em Int. J. Math. Math. Sci.}, 2004(1-4):161--193, 2004.

\bibitem{aln2}
B.~Ammann, R.~Lauter, and V.~Nistor.
\newblock Pseudodifferential operators on manifolds with a {L}ie structure at
  infinity.
\newblock {\em Ann. of Math. (2)}, 165(3):717--747, 2007.

\bibitem{Anghel1}
N.~Anghel.
\newblock An abstract index theorem on non-compact riemannian manifolds.
\newblock {\em Houston Journal of Mathematics}, 19:223--237, 1993.

\bibitem{Anghel2}
N.~Anghel.
\newblock On the index of {C}allias-type operators.
\newblock {\em Geom. Funct. Anal.}, 3, No.5:431--438, 1993.

\bibitem{AtiyahBook}
M.~Atiyah.
\newblock {\em {K}-theory}.
\newblock Benjamin, New York, 1967.

\bibitem{APS1}
M.~Atiyah, V.~Patodi, and I.~Singer.
\newblock Spectral asymmetry and {R}iemannian geometry. {I}.
\newblock {\em Math. Proc. Cambridge Philos. Soc.}, 77:43--69, 1975.

\bibitem{a-sI}
M.~Atiyah and I.~Singer.
\newblock The index of elliptic operators {I}.
\newblock {\em Ann. of Math.}, 87:484--530, 1968.

\bibitem{Baer2}
C.~B{\"a}r.
\newblock Metrics with harmonic spinors.
\newblock {\em Geom. Funct. Anal.}, 6(6):899--942, 1996.

\bibitem{Baer1}
C.~B{\"a}r and P.~Schmutz.
\newblock Harmonic spinors on {R}iemann surfaces.
\newblock {\em Ann. Global Anal. Geom.}, 10(3):263--273, 1992.

\bibitem{BottSee}
R.~Bott and R.~Seeley.
\newblock Some remarks on the paper of {C}allias.
\newblock {\em Comm. Math. Phys}, 62:235--245, 1978.

\bibitem{BM}
J.~Br\"{u}ning and H.~Moscovici.
\newblock L2-index for certain {D}irac-{S}chr\"{o}dinger operators.
\newblock {\em Duke Math. J.}, 66:311 -- 336, 1992.

\bibitem{bunke}
U.~Bunke.
\newblock A {K}-theoretic relative index theorem and {C}allias-type {D}irac
  operators.
\newblock {\em Math. Annalen}, pages 241--279, 1995.

\bibitem{bunkeSchick}
U.~Bunke and T.~Schick.
\newblock Smooth {$K$}-theory.
\newblock {\em Ast\'erisque}, (328):45--135 (2010), 2009.

\bibitem{Callias}
C.~Callias.
\newblock Axial anomalies and index theorems on open spaces.
\newblock {\em Comm. Math. Phys.}, 62:213--234, 1978.

\bibitem{carvalho}
C~Carvalho.
\newblock A {K}-theory proof of the cobordism invariance of the index.
\newblock {\em {K}-theory}, 36(1-2):1--31, 2005.

\bibitem{Cordes}
H.~Cordes and R.~McOwen.
\newblock The {$C\sp*$}-algebra of a singular elliptic problem on a noncompact
  {R}iemannian manifold.
\newblock {\em Math. Z.}, 153(2):101--116, 1977.

\bibitem{CordesBook}
H.O. Cordes.
\newblock {\em Spectral theory of linear differential operators and comparison
  algebras}.
\newblock London Mathematical Society, Lecture Notes Series 76. Cambridge
  University Press, Cambridge - London - New York, 1987.

\bibitem{CrainicFernandes}
M.~Crainic and R.~Fernandes.
\newblock Integrability of {L}ie brackets.
\newblock {\em Ann. of Math. (2)}, 157(2):575--620, 2003.

\bibitem{FH3}
J.~Fox and P.~Haskell.
\newblock Index theory for perturbed {D}irac operators on manifolds with
  conical singularities.
\newblock {\em Proc. Amer. Math. Soc.}, 123:2265 -- 2273, 1995.

\bibitem{FH2}
J.~Fox and P.~Haskell.
\newblock Comparison of perturbed {D}irac operators.
\newblock {\em Proc. Amer. Math. Soc.}, 124(5):1601 -- 1608, 1996.

\bibitem{FoxHaskell1}
J.~Fox and P.~Haskell.
\newblock The {A}tiyah-{P}atodi-{S}inger theorem for perturbed {D}irac
  operators on even-dimensional manifolds with bounded geometry.
\newblock {\em New York J. Math.}, 11:303--332, 2005.

\bibitem{Georgescu}
V.~Georgescu and A.~Iftimovici.
\newblock Crossed products of {$C^\ast$}-algebras and spectral analysis of
  quantum {H}amiltonians.
\newblock {\em Comm. Math. Phys.}, 228(3):519--560, 2002.

\bibitem{Gilkey}
P.~Gilkey.
\newblock {\em Invariance theory, the heat equation, and the {A}tiyah-{S}inger
  index theorem}, volume~16 of {\em Studies in advanced mathematics}.
\newblock CRC Press, 1995.

\bibitem{HunsickerGrieser}
D.~Grieser and E.~Hunsicker.
\newblock Pseudodifferential operator calculus for generalized {$\Bbb Q$}-rank
  1 locally symmetric spaces. {I}.
\newblock {\em J. Funct. Anal.}, 257(12):3748--3801, 2009.

\bibitem{Hitchin}
N.~Hitchin.
\newblock Harmonic spinors.
\newblock {\em Advances in Math.}, 14:1--55, 1974.

\bibitem{Horm}
L.~H{\"o}rmander.
\newblock {\em The analysis of linear partial differential operators III}.
\newblock Springer-Verlag, Berlin, Heidelberg, New York, Tokyo, 1985.

\bibitem{KaroubiBook}
M.~Karoubi.
\newblock {\em {K}-Theory - An Introduction}.
\newblock Springer-Verlag, Berlin, Heidelberg, 1978.

\bibitem{Kottke}
C.~Kottke.
\newblock An index theorem of callias type for pseudodifferential operators.
\newblock {\em Journal of K-Theory}, FirstView:1--31, 2011.

\bibitem{LMN1}
R.~Lauter, B.~Monthubert, and V.~Nistor.
\newblock Pseudodifferential analysis on continuous family groupoids.
\newblock {\em Doc.\ Math.}, 5:625--655 (electronic), 2000.

\bibitem{LMN}
R.~Lauter, B.~Monthubert, and V.~Nistor.
\newblock Pseudodifferential analysis on continuous family groupoids.
\newblock {\em Doc. Math.}, 5:625--655 (electronic), 2000.

\bibitem{LMor2}
R.~Lauter and S.~Moroianu.
\newblock Fredholm theory for degenerate pseudodifferential operators on
  manifolds with fibered boundaries.
\newblock {\em Comm. Part. Diff. Eqs,}, 26(1 \& 2):223--283, 2001.

\bibitem{LN}
R.~Lauter and V.~Nistor.
\newblock Analysis of geometric operators on open manifolds: a groupoid
  approach.
\newblock In {\em Quantization of singular symplectic quotients}, volume 198 of
  {\em Progr. Math.}, pages 181--229. Birkh\"auser, Basel, 2001.

\bibitem{LawsonBook}
H.~Lawson and M.-L. Michelsohn.
\newblock {\em Spin Geometry}.
\newblock Princeton University Press, Princeton, New Jersey, 1989.

\bibitem{MelroseScattering}
R.~B. Melrose.
\newblock {\em Geometric scattering theory}.
\newblock Stanford Lectures. Cambridge University Press, Cambridge, 1995.

\bibitem{NistorDiracEquiv}
V.~Nistor.
\newblock On the kernel of the equivariant {D}irac operator.
\newblock {\em Ann. Global Anal. Geom.}, 17(6):595--613, 1999.

\bibitem{Nistor}
V.~Nistor.
\newblock Groupoids and the integration of {L}ie algebroids.
\newblock {\em J. Math. Soc. Japan}, 52:847--868, 2000.

\bibitem{NWX}
V.~Nistor, A.~Weinstein, and P.~Xu.
\newblock Pseudodifferential operators on groupoids.
\newblock {\em Pacific J. Math.}, 189:117--152, 1999.

\bibitem{nistorHighInd}
Victor Nistor.
\newblock Higher index theorems and the boundary map in cyclic cohomology.
\newblock {\em Doc. Math.}, 2:263--295 (electronic), 1997.

\bibitem{parenti}
C.~Parenti.
\newblock Operatori pseudodifferentiali in {$\RR^n$} e applicazioni.
\newblock {\em Annali Mat. Pura ed App.}, 93:391--406, 1972.

\bibitem{rade}
J.~Rade.
\newblock {C}allias' index theorem, elliptic boundary conditions, and cutting
  and pasting.
\newblock {\em Commun. Math. Phys.}, 161:51--61, 1994.

\bibitem{Ren}
J.~Renault.
\newblock {\em A groupoid approach to {$C^{\ast} $}-algebras}, volume 793 of
  {\em Lecture Notes in Mathematics}.
\newblock Springer, Berlin, 1980.

\bibitem{Tay2}
M.~Taylor.
\newblock {\em Partial differential equations. {II}}, volume 116 of {\em
  Applied Mathematical Sciences}.
\newblock Springer-Verlag, New York, 1996.
\newblock Qualitative studies of linear equations.

\end{thebibliography}

\end{document}